\documentclass[12pt]{amsart}
\usepackage[pagebackref,colorlinks=true,linkcolor=blue,citecolor=blue]{hyperref}
\usepackage{amsmath}
\usepackage{bm}
\usepackage{adjustbox,lipsum}
\usepackage{amsfonts}
\usepackage{amssymb}
\usepackage{ mathrsfs }
\usepackage{array}
\usepackage{xfrac}

\usepackage{setspace}
\usepackage{fancyhdr}
\usepackage{tikz}
\usepackage{tikz-cd}
\usepackage{color}
\usepackage{mathtools}
\usepackage{verbatim}
\usepackage{setspace}
\usepackage{stackengine}
\usepackage{stmaryrd}
\stackMath

\usepackage{amsmath,amsthm,amssymb, amscd, tikz, tikz-cd}
\usepackage{graphicx,enumerate,stmaryrd}
\usepackage[all,cmtip]{xy}
\usetikzlibrary{decorations.markings}
\tikzstyle{vertex}=[circle, draw, inner sep=0pt, minimum size=6pt]

\usepackage{comment}
\usepackage[margin=1in]{geometry}
\usepackage{verbatim}

\usepackage[shortlabels]{enumitem}

\newtheorem{thm}{Theorem}[section]
\newtheorem{prop}[thm]{Proposition}

\newtheorem{conj}{Conjecture}[section]
\newtheorem{definition}[thm]{Definition}

\usepackage[utf8]{inputenc}

\title{Parabolic Kazhdan-Laumon and the Kloosterman Fourier Transform for Quadric Cones}
\author{Aaron Slipper}
\email{aaron.slipper@duke.edu}

\usepackage[square,sort,comma,numbers]{natbib}
\usepackage{graphicx}

\def\G{{\mathbb G}}
\def\S{{\mathcal S}}

\def\Id{{\mathrm {Id}}}
\def\ker{{\mathrm {ker}\,}}

\def\tr{{\mathrm {tr}}}

\def\Spec{{\mathrm {Spec}}}

\def\dim{{\mathrm{dim}}}

\def\GL{{\mathrm{GL}}}

\def\SL{{\mathrm{SL}}}
\def\PSL{{\mathrm{PSL}}}
\def\PGL{{\mathrm{PGL}}}
\def\sl{{\mathfrak{sl}}}

\def\SO{{\mathrm{SO}}}

\def\Sp{{\mathrm{Sp}}}

\def\Mat{{\mathrm{Mat}}}

\def\C{{\mathbb{C}}}
\def\F{{\mathbb{F}}}

\def\A{{\mathbb{A}}}

\def\Q{{\mathbb{Q}}}
\def\P{{\mathbb{P}}}

\def\Ind{{\mathrm{Ind}}}

\def\Std{{\mathrm{Std}}}
\def\Tr{{\mathrm{Tr}}}
\def\op{{\mathrm{op}}}

\def\ab{{\mathrm{ab}}}
\def\Kl{{\mathrm{Kl}}}

\allowdisplaybreaks

\setcitestyle{square}

\newcommand*{\vertbar}{\rule[-1ex]{0.5pt}{2.5ex}}

\begin{document}

\begin{abstract}
Let $G$ be a (split) reductive group over $\F_q$, and let $M$ be a standard Levi subgroup of $G$. Let $\mathcal{P}$ denote the set of parabolic subgroups of $G$ with Levi factor $M$. For $P$ and $P'$ in $\mathcal{P}$, we let $U = R_u(P)$ (resp., $U' = R_u(P')$) denote the unipotent radical; and we denote by $\overline{G/U}$ (resp., $\overline{G/U'}$) the affinization of the corresponding homogeneous space. Extending the work of Kazhdan-Laumon \cite{KazhLaum} and Braverman-Kazhdan \cite{BKI, BKIV} to general parabolic basic affine (or ``paraspherical") space, we propose a construction for certain intertwining operators 
$\mathcal{F}_{P',P}: \mathcal{S}(\overline{G/U}(\F_q), \C) \to \mathcal{S}(\overline{G/U'}(\F_q), \C)$ for suitable function spaces $\mathcal{S}$, defined via kernels analogous to those appearing in loc cit. We then study the extent to which these intertwiners are normalized. We show that, for opposite $(n-1)+1$ parabolics of $\SL_n$, our transform reduces to the classical linear Fourier transform; and that, for opposite unipotents in $\SL_3$ or opposite Siegel parabolics in $\Sp_4$, our transforms are given by a Fourier transform on a quadric cone (with kernel coming from a Kloosterman sum). We prove Fourier inversion for this transform on (a natural subclass of) functions on the quadric cone, establishing a finite-field analogue of the quadric Fourier transform of Gurevich-Kazhdan, Getz-Hsu-Leslie, and Kobayashi-Mano \cite{GurKazhQuad, GHL, KobMan}. 
\end{abstract}

\maketitle

\tableofcontents

\section{Introduction}

\subsection{The Basic Problem} This paper has two related objectives: 1) to propose a method for constructing a ``parabolic Kazhdan-Laumon" category by explicitly constructing ``normalized intertwiners" for general parabolics over finite fields, and 2) to construct a Fourier transform on quadric cones in even-dimensional affine space space over a finite field and prove the analogue of Fourier inversion. This latter goal can be viewed as an extension of the classical Fourier-Deligne transform on vector spaces to quadric cones. Finally, we show how quadric cone Fourier inversion is predicted by some special cases of our conjectural parabolic Kazhdan-Laumon theory.

Let $G$ be a split reductive group, defined over $k:=\F_q$. Let $T$ be a maximal torus, $P$ be a parabolic subgroup containing $T$, and $M$ a Levi factor, which we view as both a subgroup of $G$ and a quotient of $P$. We have an ``opposite" parabolic $P^{\op}$ such that $P \cap P^{\op} = M$. We let $U = R_u(P)$ be the unipotent radical of $P$, and $U^{\op} = R_u(P^{\op})$.

It is a theorem of Grosshans \cite{GrossUn} that $G/U$ is a strongly quasi-affine variety; i.e.,  the canonical map $G/U \to \Spec\, k[G/U]$ is an open embedding (``quasi-affine") and $k[G/U]$ is finitely generated (``strong"). We let $\overline{G/U}:=\Spec \, k[G/U] $, the ``affine closure" or ``affinization" of $G/U$. This is generally a singular variety, which we shall call paraspherical space.\footnote{
The homogeneous variety $G/U(P)$ is sometimes called ``parabolic basic affine space." However, since this space is neither basic nor affine, we are dissatisfied with this nomenclature. Our preferred terminology is naturally derived from Gelfand's elegant term ``horospherical space" for $G/U(B)$ ($B$ a Borel).}

This paper hopes to offer some first steps in extending the work of Kazhdan-Laumon \cite{KazhLaum} and Bezrukavnikov-Polishchuk \cite{BezPol} to general parabolic subgroups $P$. Working in the spirit of Braverman-Kazhdan \cite{BKI, BKIV} on normalized intertwining operators over \textit{local} fields  -- \cite{ BKIV} considers the ``degenerate" case of Braverman-Kazhdan spaces $G/[P,P]$ --  we hope to establish normalized intertwining operators for the ``full" paraspherical space $G/U(P)$. However, like \cite{KazhLaum} and \cite{BezPol}, will work in the finite-field and \'etale setting.

More concretely, we would like to find transforms

\[\mathcal{F}_{ji} : \mathcal{S}(\overline{G/U_i}(\F_q), \C) \to \mathcal{S}(\overline{G/U_j}(\F_q), \C)\]


\noindent between suitably-defined functions spaces $\mathcal{S}$, intertwining the natural $G \times M$-actions, such that

\begin{equation}\label{Id1}
\mathcal{F}_{ii} = \Id,
\end{equation}


\noindent and

\begin{equation}
\mathcal{F}_{kj} \circ \mathcal{F}_{ji} = \mathcal{F}_{ki}.
\end{equation}


\noindent

For much of this paper we will restrict our attention to constructing the ``longest"\footnote{This particular involution would correspond to the normalized intertwining operator giving the ``longest" element of the Weyl group in the Gelfand-Graev action \cite{GinzKazh}.} intertwiner -- an involutive\footnote{\label{invfoot}
Of course, it is slightly erroneous to speak of $\mathcal{F}$ as involutive, since it is a map between two different spaces. But, as we shall see, these spaces often can be identified with one another, in which case, $\mathcal{F}$ will indeed turn out to be involutive.} Fourier transform

\begin{equation}
\mathcal{F} := \mathcal{F}_{P^{\op}, P}: \mathcal{S}(\overline{G/U(P)}(\F_q), \C) \to \mathcal{S}(\overline{G/U(P^{\op})}(\F_q), \C),
\end{equation}


\noindent generalizing the Fourier-Deligne transform for the case of $G = \SL_2$ (see section \ref{SL_2case} below). Moreover, we would like to upgrade this to a functor between suitable categories of $\overline{\mathbb{Q}}_l$-adic Weil sheaves on $\overline{G/U}$ and $\overline{G/U^{\op}}$, corresponding to the above transformation under Grothendieck's fonctions-faisceaux. 

Indeed, in this paper we motivate and propose a specific cohomological kernel on $\widetilde{G/U(P)} \times_{G/Q} \widetilde{G/U(P)'}$, where $Q$ denotes the subgroup of $G$ generated by $U(P)$ and $U(P')$, and $\widetilde{G/U(P)}$ denotes the \textit{relative} affine closure of $G/U(P)$ over $G/Q$ (see \ref{PP'trnas}):

\vspace{3mm}

\textbf{Conjecture \ref{MainConj}. Parabolic Kazhdan-Laumon.} The restriction of the kernel in (\ref{PP'trnas}) to $G/U(P') \times_{G/Q}G/U(P)$, as $P$ ranges over all parabolics of $G$ containing $M$ (and with Levi quotient isomorphic to $M$), defines a Kazhdan-Laumon gluing datum for $G/U(P)$.

\vspace{3mm}

We conclude this section with some background. In the 1960's, Gelfand and Graev constructed a ``mysterious" action of the Weyl group on the space of differential operators $\mathcal{D}(G/U(B))$, for $B$ a Borel (see \cite{GinzKazh} for a discussion of the history). This action does \textit{not} arise from a geometric action of $W$ on $G/U$, but rather via compositions of Weyl-algebra Fourier transforms associated with simple reflections in the Weyl group. (Remarkably, the Coxeter relations hold for these operators.) Likewise, there is a corresponding $L^2(G/U)$-theory, where the simple reflections correspond to classical symplectic linear Fourier transforms (see \cite{GelfGraev}, Chapter 5, ``The Space of Horospheres" for a discussion of this action over the ad\`eles, and \cite{BKI} for the case of local fields).  Since the set of Borels containing a fixed torus is a torsor under the action of $W$, we may recast this action as a collection of $G \times T$-equivariant intertwiners $\mathcal{D}(G/U(B)) \to \mathcal{D}(G/U(B'))$ rather than as a $G \times T$-equivariant\footnote{Here the $T$-action is twisted by $W$.} action of $W$ on a single such $\mathcal{D}(G/U(B))$.

Kazhdan-Laumon \cite{KazhLaum} and Bezrukavnikov-Polishchuk \cite{BezPol} constructed an analogous $W$-action on a category of perverse sheaves -- though not quite on the usual categories $\textrm{Perv}(G/U(B))$ or $\textrm{Perv}(\overline{G/U(B)})$. Rather, ibid. construct a category ``glued" from $|W|$ copies of the category $\textrm{Perv}(G/U(B))$. These ``glued" categories are connected with Deligne-Lusztig Theory \cite{BravermanPolishchuk1998} and are the subject of much active research \cite{MortonFerguson2025KLCategoryO, MortonFerguson2025PolishchukConj}. 

In the context of $\mathcal{D}$-modules, the analogous ``glued" categories are simply the category of modules over $\mathcal{D}(G/U(B))$, and so produce nothing substantially new \cite{BBP}. However, over finite fields, Kazhdan-Laumon glued categories are usually \textit{not} isomorphic to the category of perverse sheaves on some space; the simple-reflection Fourier-Deligne transforms used to ``glue" the categories of perverse sheaves on $G/U$ do \textit{not} generate involutions on $\textrm{Perv}(G/U)$ and so do not correspond to normalized intertwiners as they do in the case of the algebras $\mathcal{D}(G/U(B))$ or function spaces on $G/U(B)$ over local fields or the ad\`eles. Rather, these Fourier transforms satisfy Hecke-type relations \cite{BezPol}; on certain subcategories of ``monodromic" sheaves in $D^{\flat}(G/U, \overline{\Q}_\ell)$, they are indeed involutive (see \cite{BravermanPolishchuk1998}, Section 2.1.3.).

In this paper, like in \cite{KazhLaum} and \cite{BezPol}, we will concern ourselves $\ell$-adic sheaves (resp., function-spaces on varieties over finite fields). However, rather than working with Kazhdan-Laumon``glued" categories, we will mostly be considering subcategories of $D^{\flat}(\overline{G/U}, \overline{\Q}_\ell)$ (resp., $\textrm{Fun}(\overline{G/U}, \C)$) over which the natural cohomological (resp, integral) kernels we construct will be involutive (i.e., normalized). (Over local fields, the analogous spaces are considered in to \cite{KazhdForm}).

Finally, we mention that \cite{ABBGM2005} and \cite{MortonFerguson2025KLCategoryO} suggest that there is a connection between Braverman-Kazhdan's Fourier transforms (over \textit{local} fields), their finite-field analogues via Kazhdan-Laumon gluing, and sheaves on the semi-infinite flag variety. One may naturally ask about parabolic analogues of this story (e.g., \cite{BezrukavnikovLosev2023JAMS}, Section 8.11). In the folklore, parabolic semi-infinite flag spaces are usually referred to as ``quarter-infinite" flag varieties. This paper can be viewed as a first step towards understanding Fourier-Deligne transforms for quarter-infinite sheaves.

\subsection{The Fourier Transform on a Quadric Cone} The central new result of this paper is a finite field Fourier transform on even quadric cones, which, we will show, agrees with the ``longest" intertwiner in the fundamental cases of opposite Borels in $\SL_3$ and opposite Siegel parabolics in $\Sp_4$. This Fourier transform is a non-linear, quadratic analogue of the Fourier transform for $\A^n$. In particular, we have:

\vspace{3mm}

\textbf{Theorem \ref{KlInv}.} Let $V$ denote affine $n$-space over $\F_q$, and let 

\[
X = \{(v, v^{\vee}) \in V \times V^* : \langle v, v^\vee\rangle = 0\}.
\]

\noindent Let $\psi: \F_q \to \C^{\times}$ be an additive character, and let $\Kl(a) := \sum_{t \in \mathbb{F}_q^{\times}} \psi\left(\frac{a}{t} + t\right)$ denote the usual Kloosterman sum. Next, define

\[
\mathcal{S}(X(\mathbb{F}_q),\C) =  \left\{f \in \text{Fun}(X(\mathbb{F}_q),\C) : \sum_{\lambda \in \mathbb{F}_q^{\times}} f(\lambda x) = 0\right\},
\] 

\noindent and let $K: X \times X \to \mathbb{A}^1 \to \C$ be given by

\[
\left((u, u^{\vee}),(v, v^{\vee})\right) \mapsto \langle u, v^{\vee}\rangle + \langle v, u^{\vee}\rangle \mapsto \Kl\left(\langle u, v^{\vee}\rangle + \langle v, u^{\vee}\rangle\right).
\]

\noindent Then defining the quadric Fourier transform on $\textrm{Fun}(X(\F_q), \C)$ via 

\[
\mathcal{F}(f)(y) = \sum_{x \in X(\mathbb{F}_q)} f(x)K(x,y),
\]

\noindent we have, for all $f \in \mathcal{S}$, 

\[
\mathcal{F}^2(f) = q^{2n}f.
\]

\vspace{3mm}

The analogous construction for local fields is due to Kobayashi-Mano, in the case of Archimedean fields \cite{KobMan}, and by Gurevich-Kazhdan \cite{GurKazhQuad} and Getz-Hsu-Leslie \cite{GHL} for general local fields. In this body of work the Fourier transform on $L^2(\textrm{Quadric Cone})$ is shown to arise from a kernel given by an explicit Bessel distribution; as we shall see, our geometric kernel is given by the Bessel (or Kloosterman) sheaf of Deligne \cite{Del, KatzGauss}. Hence our work can be thought of as a ``geometrization" of ibid.

To obtain Fourier inversion, we must restrict to those functions that lie in the kernel of the averaging operator under the natural scalar $\G_m$-action on the cone. (This is likely related to the singular geometry of the cone at the origin.) In the context of $\SL_3/U(B)$ (which is an example of a quadric cone), this $\G_m$-averaging condition follows from taking traces of Frobenius of objects in the regular-monodromic category of  Braverman-Polishchuk (\cite{BravermanPolishchuk1998}, Section 2.1.3). 

In addition, our finite-field quadric Fourier transform for functions on the $\F_q$-points of cones in \textit{even}-dimensional affine space and functions whose $\G_m$-\textit{average} under the scaling action is 0, offers an interesting complement to the results of Laumon-Letellier \cite{LLII}, who construct a Fourier transform for functions on the $\F_q$-points of the \textit{quotient stack} of a quadric cone under the $\mathbb{G}_m$-scaling action for cones in \textit{odd}-dimensional affine space. 

\subsection{Outline} The rest of this paper is structured as follows:

Section 2 will be concerned with the foundational example of the normalized intertwining operator for $\SL_2$, and general expectations we have about normalized intertwiners, following Braverman-Kazhdan \cite{BKI, BKIV}.


In section 3, we will try to extract certain general ideas from the $\SL_2$ case; this will lead us to introduce the Wang Monoid and the Slipper pairing. 


In section 4, we present our proposed construction for the Fourier transform (i.e., normalized intertwining operator) for opposite parabolics. We present Fourier inversion (i.e., involutivity) as a conjecture. We then extend our transform to all pairs of parabolics $P, P'$. 


The second half of this paper investigates the involutivity of the proposed transforms.


In section 5, we discuss the case or mirabolic subgroups of $\SL_n$. In this case, involutivity follows, without any restrictions, from finite-field Fourier inversion on a vector space.


In section 6, we discuss the central new result of this paper (Theorem \ref{KlInv}): Fourier inversion for functions on quadric cones in an even-dimensional affine space.


In section 7, we examine the case of the six Borels of $\SL_3$, showing that the quadric Fourier transform of section 6 is precisely the ``longest" intertwiner while the ``simple" reflections are indeed built from the $\SL_2$ (i.e., symplectic linear) Fourier transforms.


In section 8, we analyze the Siegel Parabolic of $\Sp_4$. We show that the quadric Fourier transform of section 6 gives us our normalized intertwiner.

\subsection{Acknowledgments} The author would like to thank his advisors, Ng\^o B\`ao Ch\^au and Kazuya Kato, as well as G\'erard Laumon and Victor Ginzburg for their inspiring counsel. He would also like to thank Barry Mazur for helpful suggestions on the manuscript, as well as Jayce Getz, Yiannis Sakellaridis, Zhilin Luo, Kevin Lin, Justin Campbell, Benedict Morrissey, Minh-Tam Trinh, Tom Gannon, Gurbir Dhillon, Calder Morton-Ferguson, Chun-Hsien Hsu, and Nikolay Grantcharov for many helpful conversations. In addition, he would like to thank an anonymous referee who made detailed comments on the manuscript and pointed out the potential connection to quarter-infinite sheaves. The author would also like to offer deep thanks to John Boyle for his indispensable help in patiently computerizing the quadric Fourier transform proposed here. It was John's empirical verifications of (\ref{casesfor}) for $\F_3$ and $\F_5$ that gave the author confidence in the results of Section 6. This material is based upon work partially supported by the National Science Foundation Graduate Research Fellowship under Grant No. 2140001 and by Duke University's Number Theory RTG under DMS-2231514.

\section{The Case of \texorpdfstring{$\text{SL}_2$}{SL2} and General Expectations}\label{SL_2case} 
\subsection{\texorpdfstring{$\SL_2$}{SL2} and the symplectic linear Fourier transform}. For motivation, we now give the primordial example of a ``normalized intertwiner" Fourier transform. In this subsection, let $G= \SL_2$. Then, up to conjugacy, we have only one parabolic, the upper-triangular Borel with it usual unipotent radical $U = U(B) = [B,B]$.

 We observe that there is a well-defined map $G/U \to \A^2$ given by:

\begin{align}\label{A2+}
    \left(\begin{matrix}
        a & b\\
        c & d
    \end{matrix}\right)\left(\begin{matrix}
        1 & *\\
        0 & 1
    \end{matrix}\right) \mapsto \left(\begin{matrix}
        a\\
        c
    \end{matrix}\right).
\end{align}


\noindent The map is open and isomorphic onto its image, which consists of the set of first columns of matrices in $\SL_2$ -- that is to say, every nonzero vector. Thus $\SL_2/U \cong \A^2 \setminus\{(0,0)\}$. Moreover, as every regular function on $\A^2 \setminus \{ (0,0)\}$ extends canonically to a regular function on $\A^2$ (a computation either verifiable directly or by appeal to the algebraic Hartog 
Lemma\footnote{Indeed, the case of $\A^2\setminus \{ (0,0)\}$ is perhaps the most well-known classical example of Hartog's Lemma.}), we see that

\[
\overline{\SL_2/U(B)} = \overline{\SL_2/[B,B]} \cong \A^2 .
\]


\noindent Observe that, in particular, $\overline{\SL_2/U}$ is smooth -- this is, in fact, quite miraculous; most affinizations of quotients $G/U(P)$ are singular.

Let  $T$ denotes the standard torus in $\SL_2$, and let $B^{\op}$ (resp. $U^{\op}$) denote the lower triangular Borel (resp., its unipotent radical).

We have a map $G/U^{\op} \to \A^2$ via a map which we shall write, somewhat suggestively, as:

\begin{align}\label{A2-}
    \left(\begin{matrix}
        a & b\\
        c & d
    \end{matrix}\right)\left(\begin{matrix}
        1 & 0\\
        * & 1
    \end{matrix}\right) \mapsto \left(\begin{matrix}
        1 & 0\\
        * & 1
    \end{matrix}\right)\left(\begin{matrix}
        a & b\\
        c & d
    \end{matrix}\right)^{-1} =  \left(\begin{matrix}
        1 & 0\\
        * & 1
    \end{matrix}\right)\left(\begin{matrix}
        d & -b\\
        -c & a
    \end{matrix}\right) \mapsto \left(\begin{matrix}
        d & -b\\
    \end{matrix}\right).
\end{align}


\noindent And, of course, this exhibits $\SL_2/U^{\op}$ as another $\A^2\setminus \{(0,0)\}$, with affine closure $\A^2$ as before.

Braverman-Kazhdan's Fourier transforms in \cite{BKI, KazhdForm} are built from the observation that there is a $G$-invariant symplectic duality between these two $\A^2$'s. Thus the classical (symplectic) linear Fourier transform can be used to provide normalized intertwiners in this case. 

Let us choose once and for all $\psi: \F_q \to \C$, an additive character. We consider the pairing, defined as the composition: 

\begin{equation}\label{pairings12}
\overline{\SL_2/U}(\F_q) \times \overline{\SL_2/U^{\op}} (\F_q)
\to \A^1(\F_q) \xrightarrow{\psi} \C 
\end{equation}
\[
\left(\left(\begin{matrix}
    a & b \\
    c & d
\end{matrix}\right)\, \text{mod} \,\, U, \left(\begin{matrix}
    a' & b' \\
    c' & d'
\end{matrix}\right)\,\text{mod}\,\, U^{\op}\right)\mapsto ad' - cb' \mapsto \psi(ad' - cb').
\]


\noindent This $\C$-valued pairing provides the integral kernel of the classical (symplectic) Fourier transform:

\begin{equation}\label{SL2Four}
 \mathcal{F}(f)\left[\left(\begin{matrix}
     b\\
     d
 \end{matrix}\right)\right] = q^{-1}\sum_{\left(\begin{matrix}
     a\\
     c
 \end{matrix}\right)\in \A^2(\F_q)} f\left[\left(\begin{matrix}
     a\\
     c
 \end{matrix}\right)\right]\psi(ad - bc).
\end{equation}

Furthermore, this may be upgraded to sheaves in the form of the Fourier-Deligne transform, via ``push-pull" along the diagram: 

\[
\begin{tikzcd}
             & \mathbb{A}^2\times \mathbb{A}^2 \arrow[ld, "\pi_1"'] \arrow[rd, "\pi_2"] \arrow[r, "{\langle -,-\rangle}"] & \mathbb{A}^1 \\
\mathbb{A}^2 &                                                                                                            & \mathbb{A}^2.
\end{tikzcd}
\]


\noindent If $\mathscr{S} \in D^{\flat}(\A^2, \overline{\mathbb{Q}_\ell})$ is  bounded derived $\overline{\mathbb{Q}_\ell}$-adic Weil sheaf, then: 

\begin{equation}\label{FourSheavesCl}
    \mathcal{F}(\mathscr{S}) := {\pi_1}_!\left(\pi_2^*\mathscr{S} \otimes \langle-,-\rangle^*(\mathcal{L}_{\psi})\right)[2](1),
\end{equation}


\noindent where $\mathcal{L}_\psi$ is the Artin-Schreier sheaf associated to $\psi$.

The fact that these transforms are involutive (equivalently, that they define \textit{normalized} intertwiners) is then an immediate consequence of Fourier inversion.\footnote{The fact that $\mathcal{F}^2 = 1$ as opposed to the usual $\mathcal{F}^4 = 1$ as one might recall from the classical Fourier transform is due to the symplectic nature of the pairing $\langle-,-\rangle$. Note also the factor of $q^{-1}$ in (\ref{SL2Four}), and the Tate twist in (\ref{FourSheavesCl}), is calibrated so as to make $\mathcal{F}^2 = \Id$.}

\subsection{Some Expectations}

For the present section we are going to consider only the \textit{longest} intertwiner, $\mathcal{F}_{P^{\op}P}$, for $P$ an arbitrary parabolic with Levi factor $M\subset G$. We expect it to satisfy a few desiderata:


\noindent \textbf{Ansatz 1)}. $\mathcal{F}$ must intertwine the natural $G \times M$ action

\[
\left((g,m)[f]\right)(x) = f(g^{-1}xm)
\]


\noindent on the two function spaces. Observe that this ansatz precludes the transformation from coming from an underlying map of spaces: there is no map $\overline{G/U} \to \overline{G/U^{\op}}$ intertwining the two standard $G \times M$-actions.


Now, there is a natural candidate for a transform satisfying Ansatz 1); what Braverman-Kazhdan call the \textbf{Radon Transform}:

\[\mathcal{R}: \mathcal{S}(G/U(P)\, (\F_q), \C) \to \mathcal{S}(G/U(P^{\op})\, (\F_q), \C).
\]


\noindent It is defined\footnote{Braverman-Kazhdan work only with local fields in \cite{BKI, BKIV}, so everything here is transferred to finite fields by analogy.  Sometimes, for instance, we use language of ``measures", even though over finite fields, these will mean either functions or sheaves.} by

\begin{equation}\label{Radon}
f \mapsto \mathcal{R}(f) := \left\{ x \mapsto \sum_{\overline{u} \in U^{\op}} f(x\overline{u})\right\}.
\end{equation}


This is equivalent to the ``pull-push" of the function $f$ along the roof:\footnote{Note that the Radon transform is defined on the homogeneous spaces $G/U$, \textit{not} their affine closures $\overline{G/U}$. We could, of course, extend it to affine closures by replacing $G$ in the correspondence (\ref{RadCorr}) by the closure of the diagonal embedding of $G$ in $\overline{G/U^{\text{op}}}\times\overline{G/U}$. In this case (\ref{Radon}) would need to be appropriately adjusted.}

\begin{equation}\label{RadCorr}
\begin{tikzcd}
                & G \arrow[ld] \arrow[rd] &     \\
G/U^{\text{op}} &                         & G/U.
\end{tikzcd}
\end{equation}


\noindent Note that the Radon transform manifestly intertwines the $G \times M$ actions (since $M$ normalizes $U$). However, the transform is not involutive; that is to say, in general,

\begin{equation}\label{RR}
\mathcal{R}'\mathcal{R}(f)(x) = \left\{ x \mapsto \sum_{u \in U} \sum_{\overline{u} \in U^{\op}} f(x\overline{u}u)\right\} \ne f(x).
\end{equation}


Indeed, the Gindikin-Karpelevich formula \cite{GindKapI, GindKapII} shows that if the right $M$-orbit of $f$ generates an irreducible representation of $M$, then the LHS of (\ref{RR}) will turn out to equal $f$ times a suitable product of $\gamma$-factors.\footnote{In the finite-field case, the analogue of the $\Gamma$-function is a Gauss sum. See \cite{Carter}, chapter 10, for more information on these kinds of intertwiners over finite fields.}

On the other hand, we are considering paraspherical spaces $G/U(P)$. By analogy with Braverman-Kazhdan's work, we should convolve the Radon transform by a measure on the Levi $M$:

\begin{equation}\label{RadonNorm}
f \mapsto \mathcal{F}(f) := \left\{ x \mapsto \sum_{\overline{u} \in U^{\op}, m \in M} f(xm\overline{u})J(m)\right\}
\end{equation}


\noindent that will kill the $\gamma$-factors. Notice that if such a transform is to intertwine the $M$-action, the function $J$ must be central on $M$; i.e., $M$-conjugacy-invariant.\footnote{For a $M$ a torus, as in the case where $P = B$ is a Borel, this property comes for free; tori are commutative.}

However, proceeding exactly as in (\ref{RadonNorm}) is, as we shall see, unsatisfactory. Rather we shall assume, following (\ref{SL2Four}), that we have:


\noindent \textbf{Ansatz 2)}. The Fourier transform $\mathcal{F}$ takes the form:

\begin{equation}\label{GenFour}
\mathcal{F}(f) (y) = \sum_{x \in \overline{G/U}(\F_q)} f(x) \langle x, y \rangle 
\end{equation}


\noindent for all $y \in \overline{G/U^{\op}}$, where $\langle -, - \rangle: \overline{G/U}(\F_q) \times \overline{G/U^{\op}}(\F_q) \to \C$ is a $\C$-valued pairing between $\overline{G/U}$ and $\overline{G/U^{\op}}$. The pairing $\langle -, - \rangle$ will depend upon choice of an additive character $\psi: \F_q \to \C$. In the case where $G = \SL_2$ and $P$ is the Borel, $\mathcal{F}$ is given by (\ref{SL2Four}). We also expect that $\langle -,- \rangle$ arises as the function (trace of Frobenius) of a natural sheaf on $\overline{G/U} \times \overline{G/U^{\op}}$.


Now we can see why we ought not to proceed directly as in (\ref{RadonNorm}): the (left) action of $TU^{\op}$ on $\SL_2/U$ is not transitive; however, the pairing $\psi\circ \langle -,- \rangle$ in (\ref{SL2Four}) is nonzero everywhere on $\overline{G/U} \times \overline{G/U^{\op}}$. (In particular, it is nonzero on the complement of $yTU^{\op}U$ for any $y$). Hence, even for $G = \SL_2$, we cannot write the symplectic Fourier transform  (\ref{SL2Four}) in the form (\ref{RadonNorm}).

In the context of \cite{BKI, BKIV}, this consideration is not significant, since Braverman-Kazhdan are dealing with integration over a local field $F$. As the complement of the ``big" $TU^{\op}$-orbit has measure 0, a measure defined on the ``big" orbit (via convolving the Haar measure on $U^{\op}$ with some measure on $T$) completely defines a smooth measure on $G/U(F)$. Similarly, Braverman-Kazhdan do not have to contend with the difference between $G/U(F)$ and $\overline{G/U}(F)$, since the complement of the former in the latter has measure 0 over local fields.

Over finite fields, however, the rational points of Zariski-closed subsets have positive measure. Thus there are differences between transforms of type (\ref{RadonNorm}) as opposed to (\ref{GenFour}), and between functions on $G/U$ as opposed to functions on $\overline{G/U}$. Appeal to the $\SL_2$ case shows that we should prefer transforms of type (\ref{GenFour}) and functions on $\overline{G/U}$ as opposed to functions on $G/U$. 

However, we might hope that some vestige of the ``normalized" transform of Braverman-Kazhdan can persist in the finite field setting. In particular:

\noindent \textbf{Ansatz 3)}. If $y \in xUM U^{\op}$, then there should exist a function $M \to \C$ such that, if $y =x u m  \overline{u}$, then $\langle x, y \rangle = J(m)$. The function $J$ should be invariant under (stable\footnote{Here \textit{stable conjugacy} means conjugate over the algebraic closure. I.e., $x$ and $y$ in $M(\F_q)$ are stably conjugate iff there exists $m \in M\left(\overline{\F_q}\right)$ such that $mxm^{-1} = y$. Conjugacy within $M(\F_q)$ will be called \textit{rational conjugacy}. The rationale for insisting upon stable, as opposed to rational conjugacy, also comes from appeal to the work of Braverman-Kazhdan \cite{BKI, BKII}.}) $M$-conjugacy. 


Thus we are insisting that the kernel $\langle -, -\rangle$ restrict to the form (\ref{RadonNorm}) when $x$ and $y$ are in the most ``generic" relative position.


Finally, we record as a last ansatz the most important property that we expect of our Fourier transform. The whole point of normalization is to obtain transforms between function spaces that compose well; in our case (dealing with only $P$ and $P^{\op}$), this means involutivity.


\noindent \textbf{Ansatz 4)}. (Involutivity.\footnote{As noted in footnote \ref{invfoot}, if we identify the spaces $\overline{G/U}$ and $\overline{G/U^{\op}}$ (which are isomorphic as varieties, even if no such isomorphism intertwines the standard $G \times M$ actions), the involutive transform $\mathcal{F}$ is then seen to \textit{square} to one. Recall that the standard Fourier transform has \textit{fourth power} one. So perhaps it would be more apropos to call our transform a \textit{Two}rier as opposed to a \textit{Four}ier transform.})
Swapping the roles of $P$ and $P^{\op}$, we have the transform $\mathcal{F}^{\op}: \mathrm{Fun}(\overline{G/U(P^{\op})}(\F_q), \C) \to \mathrm{Fun}(\overline{G/U(P)}(\F_q), \C)$. Then there exists a natural subspace $\mathcal{S} \subset \mathrm{Fun}(\overline{G/U(P)}(\F_q), \C)$ such that for all $f \in \mathcal{S}$,

\[
\mathcal{F}^{\op} \circ \mathcal{F}(f) = f.
\]



\subsection{\texorpdfstring{$\SL_2$}{SL2} Revisited}

It is straightforward to see that the symplectic Fourier transform (\ref{SL2Four})  accomplishes Ansatz 1) and 2), while Ansatz 4), as has already been pointed out, follows from classical Fourier inversion. We would like to understand how it  accomplishes Ansatz 3). To this end, let us scrutinize the first ``$\F_q$" duality (with values in $\A^1$) between $\overline{\SL_2/U}$ and $\overline{\SL_2/U^{\op}}$ in (\ref{pairings12}).

Observe that we have an isomorphism between the double-unipotent quotient and $\A^1$:

\begin{equation}\label{doubleuinv}
    U^{\op}\backslash SL_2 / U \cong \A^1,
\end{equation}
\[
    \left(\begin{matrix}
        1 & 0\\
        * & 1
    \nonumber \end{matrix}\right) \left(\begin{matrix}
        a & b\\
        c & d
    \nonumber \end{matrix}\right)\left(\begin{matrix}
        1 & *\\
        0 & 1
    \nonumber \end{matrix}\right) \mapsto a.
\]


\noindent In particular, we see that the function $a: \SL_2 \to \A^1$ freely generates all double-unipotent invariant functions on $\SL_2$. And observe that $\A^1 = \Spec\, k[a]$ has a multiplicative monoid structure (under multiplication) such that $(\A^1)^{\times} \cong \mathbb{G}_m \cong T$.

We see, therefore, that we may write the $\A^1$-pairing via

\begin{align*}
    \mathfrak{S}: SL_2 / U \times \SL_2/U^{\op}  &\to U^{\op}\backslash SL_2 / U\cong \A^1;\\
    \nonumber (gU, hU^{\op}) &\mapsto U^{\op} h^{-1} g U,
\end{align*}


\noindent which is manifestly $G$-invariant since $U^{\op}h^{-1}gU = U^{\op}(g'h)^{-1}(g'g)U$.

Indeed: 

\[\left(\begin{matrix}
    a' & b' \\
    c' & d'
\end{matrix}\right)^{-1}\left(\begin{matrix}
    a & b \\
    c & d
\end{matrix}\right) = \left(\begin{matrix}
    ad'-cb' & * \\
    * & *
\end{matrix}\right),
\]


\noindent so that 

\begin{equation}\label{sl2pairing}
\left(\left(\begin{matrix}
    a & b \\
    c & d
\end{matrix}\right)\, \text{mod} \,\, U, \left(\begin{matrix}
    a' & b' \\
    c' & d'
\end{matrix}\right)\,\text{mod}\,\, U^{\op}\right)\mapsto ad' - cb',
\end{equation}


\noindent as desired. Note, furthermore, that if the elements of $G/U(B)$ and $G/U(B')$ come from the same element $g$ of $G$ (i.e., if $a = a',\ldots, d=d'$), then the pairing yields $\det(g) = 1$.

What is the geometric meaning of this pairing? If we are given two elements $gU$ and $hU^{\op}$ of $G/U$ and $G/U^{\op}$, respectively, then, generically, they will pair to an element of $T = (\A^1)^{\times} \subset \A^1$. This will be the unique element $t \in T$ such that we have

\[
\begin{tikzcd}
         & g \arrow[ld, maps to] \arrow[rd, maps to] &                \\
gU &                                           & htU^{\text{op}}
\end{tikzcd}
\]


\noindent under the projections

\[
\begin{tikzcd}
       & G \arrow[ld] \arrow[rd] &                    \\
G/U &                         & G/U^{\text{op}}.
\end{tikzcd}
\]

\noindent Thus saying that the $\C$-valued pairing $\langle -,- \rangle: \overline{G/U} \times \overline{G/U^{\op}} \to \C$ factors through $\mathfrak{S}$ via

\begin{equation}\label{SPairSL_2}
\overline{\SL_2/U} \times \overline{\SL_2/U^{\op}} 
\xrightarrow{\mathfrak{S}} \A^1 \xrightarrow{\psi} \C 
\end{equation}


\noindent is precisely the compromise of a transform like (\ref{GenFour}) with one like (\ref{RadonNorm}): we see that, generically, when $gU$ intersects the $T$-orbit of $hU^{\op}$, the unique $t$ bringing $hU^{op}$ into contact with $gU$ determines $\langle gU, hU^{\op} \rangle = \psi(t): = J(t)$, adopting the notation of Ansatz 3).

Thus letting $\psi \circ \mathfrak{S} := \langle -, - \rangle$ in (\ref{GenFour}) gives us the involutive Fourier transform (\ref{SL2Four}) of Braverman-Kazhdan. This will be our model for constructing the general Fourier transform. 


\section{The Wang Monoid and the Pair of Slipper's}

We now are going to generalize what we have found for $\SL_2$. Note that our ultimate goal is to construct a kernel sheaf on $\overline{G/U} \times \overline{G/U^{\op}}$ such that the corresponding integral transform is involutive. What careful study of the $\SL_2$ example reveals is that, for it to resemble a ``normalization" of the Radon transform, this kernel sheaf should be pulled back from some map of the form:

\[
\overline{G/U} \times \overline{G/U^{\op}} \xrightarrow{\mathfrak{S} } \overline{M}
\]


\noindent where $\overline{M}$ is variety over $\F_q$ ``enlarging" $M$. In this section, we will provide a candidate space $\overline{M}$ and map $\mathfrak{S}$.

First let us turn our attention to the space $\overline{M}$. In the case of $\SL_2$, we saw that it turned out to be $\A^1$, which contains $\mathbb{G}_m = T$ as an open subset. Thus we might imagine that $\overline{M}$ is a reductive monoid whose group of invertible elements is $M$.

\subsection{The Wang Monoid} There is a natural candidate $\overline{M}$, which we will call the Wang monoid as Wang's article \cite{Wang} was the first to isolate and scrutinize it specifically (though see also \cite{ArTi}).  We will offer just one very quick and elegant definition of this monoid which is particularly well-suited to our purposes.\footnote{See also \cite{NgTak} for more on the role of reductive monoids and their relation to Braverman-Kazhdan theory. The particular appearance of the Wang monoid in Braverman-Kazhdan theory appears to be new, however.}

Consider the actions of $U(P)$ and $U(P^{\op})$ on $G$ -- on the right and on the left respectively. These actions induce actions on the ring of functions, $k[G]$. Then we let

\[
\overline{M} := \Spec \, k[G]^{U(P^{\op}) \times U(P)} .
\]


\noindent Wang \cite{Wang} shows that this is an affine, normal, algebraic monoid, with group of units $M$; he also presents its combinatorial description via its Renner Cone.

\subsection{The Slipper Pairing} Now we will turn our attention to the map $\mathfrak{S}$. We will call this map the Slipper Pairing (because of 
the author's surname, and 
the natural predilection for Slippers to come in pairs), and let

\[
\mathfrak{S}: \overline{G/U(P)} \times \overline{G/U(P^{\op})} \to \overline{M}
\]


\noindent be given as follows. We first define the restriction:

\begin{equation}
\mathfrak{S}: G/U(P) \times G/U(P^{\op}) \to \overline{M},
\end{equation}

\[
\left(gU(P), hU(P^{\op})\right) \mapsto U(P^{\op})h^{-1} g U(P),
\]


\noindent and then extend to affine closures. Notice that, generically, $\mathfrak{S}(gU(P), hU(P^{\op})) \in M$, in which case it is the unique element $m \in M$ such that we have 

\[
\begin{tikzcd}
         & g \arrow[ld, maps to] \arrow[rd, maps to] &                \\
gU &                                           & hmU^{\text{op}}
\end{tikzcd}
\]


\noindent under the maps

\[
\begin{tikzcd}
       & G \arrow[ld] \arrow[rd] &                    \\
G/U(P) &                         & G/U(P^{\text{op}}).
\end{tikzcd}
\]

\noindent That is to say, it the unique element of $m \in M$ whose right action on $hU^{\op}$ makes $gU$ and $hmU^{\op}$ ``come from" the same element of $G$. This comports well with the $\SL_2$ case and with Ansatz 3).

\section{The Construction of the Non-Abelian Fourier Transform}

\subsection{Braverman-Kazhdan Normalization} We see that the kernel of our Fourier transform should arise from some sheaf $\mathfrak{J}$ on $\overline{M}$. Let us call corresponding function (given by trace of Frobenius)

\[
J: \overline{M} \to \C,
\]


\noindent provided by choosing some isomorphism $\overline{\mathbb{Q}_\ell} \simeq \C$.

As we have already discussed, if this kernel is to preserve $G \times M$-equivariance of the Fourier transform, $\mathfrak{J}$ must be equivariant under the conjugation action of $M$ on $\overline{M}$. (Recall that $M = \overline{M}^{\times}$ so conjugation by $M$ is well-defined.)

At the level of functions we desire $J: \overline{M} \to \C$ such that: 1) $J$ is $\textrm{Ad}$-$M$ invariant, 2) $J$ depends upon an additive character $\psi$, and 3) when $P =B$, a Borel, we expect that $J$, at least when restricted to $M = T$, should be analogous to those normalizing factors that appear in Braverman-Kazhdan \cite{BKIV} (since, in the case $P= B$, we have $[P,P] = U(P)$). Moreover, we hope that, even in the general case, $J$ will have at least some kinship with the normalizing factors described in \cite{BKIV} for the ``degenerate" case of $G/[P,P]$.

The natural candidate, also arising from the work of Braverman-Kazhdan, comes from the theory of $\gamma$-sheaves for reductive groups \cite{BKIII}. These are designed to yield $\gamma$-factors under convolution, which we expect will kill the $\gamma$-factors arising in the Gindikin-Karpelevich formula. Over local fields, the $\gamma$-distributions \cite{BKI} are still conjectural; for finite fields, on the other hand, the analogous $\gamma$-\textit{sheaves} \cite{BKIII} are completely defined using the theory of perverse sheaves.

\subsection{The Definition of the \texorpdfstring{$\gamma$}{gamma}-sheaf} The BK $\gamma$-sheaf $\Phi_{G, \rho, \psi}$ is a sheaf on $G$, a reductive group, determined by a representation $\rho$ of the Langlands dual group $G^{\vee} \to \GL_n(\overline{\mathbb{Q}_l})$ and an additive character $\psi$. These sheaves are constructed over $\overline{\mathbb{Q}_l}$, so we will pick an isomorphism $\C \cong \overline{\mathbb{Q}_l}$ to translate this into $\C$-valued functions.

$\Phi_{G, \rho, \psi}$ is constructed as follows\footnote{Here we are following Cheng-Ng\^o \cite{ChNg}.}: let $T^{\vee} \subset G^{\vee}$ be a maximal torus. Restricting $\rho$ to $T^{\vee}$ we may diagonalize the representation $\rho$ with respect to the weights of $T^{\vee}$, which gives us a multi-set of characters $\{\lambda_i\}$, $\lambda_i: T^{\vee} \to \mathbb{G}_m$. By duality, these define a multi-set of cocharacters $\lambda_i: \mathbb{G}_m \to T$. We consider the diagram:

\[
\begin{tikzcd}
             & \mathbb{G}_m^n \arrow[rd, "p_{\underline{\boldsymbol{\lambda}}}"] \arrow[ld, "\text{Tr}"'] &   \\
\mathbb{A}^1 &                                                                                            & T
\end{tikzcd}
\]


\noindent where $\text{Tr}$ is given by summing the coordinates in each factor of $\mathbb{G}_m$, and $p_{\underline{\boldsymbol{\lambda}}}$ is given by the product of the maps $\lambda_i: \mathbb{G}_m \to T$. (Note that there are $n = \dim \,\rho$ of the $\lambda_i$, counting each with multiplicity; thus each factor of $\mathbb{G}_m$ corresponds to a $\lambda_i$ and we define $p_{\underline{\boldsymbol{\lambda}}}: (x_1, \ldots , x_n) \mapsto \lambda_1(x_1) \cdots \lambda_n(x_n) \in T$.) Let $\mathcal{L}_\psi$ be the Artin-Schreier sheaf on $\A^1$ associated to 
the additive character $\psi$. We define the sheaf $\Phi_{T, \rho, \phi}$ on $T$ via 

\begin{equation}\label{GammaTorus}
\Phi_{T, \rho, \psi} : = p_{\underline{\boldsymbol{\lambda}}\,!} \text{Tr}^*(\mathcal{L}_\psi) [n]\left(\frac{n}{2}\right).
\end{equation}


\noindent When the $\lambda_i$ satisfy a technical condition called $\sigma$-positivity (which will prevail in the cases of interest to us) this sheaf is a perverse local system over its support (which is the subtorus of $T$ given by the image of $p_{\underline{\boldsymbol{\lambda}}}$). It is sometimes called the Kloosterman sheaf, the hypergeometric sheaf, or the $\gamma$-sheaf for tori. One can construct a $W$-equivariant structure on $\Phi_{T, \rho, \psi}$, twisting the ``obvious" $W$-action by an appropriate sign character (see \cite{ChNg} for details).

Next, one considers the Grothendieck-Springer simultaneous resolution:

\[
\begin{tikzcd}
\widetilde{G} \arrow[d, "\widetilde{q}"'] \arrow[r, "\widetilde{c}"] & T \arrow[d, "q"] \\
G \arrow[r, "c"']                                                    & T\sslash W      
\end{tikzcd}
\]


\noindent where 

\[
\widetilde{G} = \{(g, hB) \in G \times G/B \,:\, h^{-1}gh \in B \},
\]


\noindent $c$ is the Steinberg morphism, $\widetilde{c}: (g, hB) \mapsto \pi(h^{-1} g h)$ (where $\pi:  B \twoheadrightarrow T$ is the canonical projection), $q$ is the obvious projection, and $\widetilde{q}: (g, hB) \mapsto g$. Notice that if we restrict everything to the regular semisimple locus (of both $T$ and $G$), the diagram becomes Cartesian: the upper row is a $W$-torsor over the lower.

Now we consider the sheaf on $G$ given by pull-pushing $\Phi_{T, \rho, \psi}$ along the Grothendieck-Springer fibration:

\[
\Ind_T^G (\Phi_{T, \rho, \psi}) := \widetilde{q}_{\,!}\,\widetilde{c}\,^*(\Phi_{T, \rho, \psi})[d]\left(\frac{d}{2}\right),
\]


\noindent where $d = \dim \,G - \dim \,T$.\footnote{This normalization differs by a Tate twist from \cite{ChNg}.} Because the Grothendieck-Springer map $\widetilde{q}$ is small and proper, we see that $\Ind_T^G (\Phi_{T, \rho, \psi})$ is a perverse sheaf on $G$. Let $j^{\text{rss}}: G^{\text{rss}} \to G$ denote the open inclusion of the regular semisimple locus of $G$. We see that

\[
\Ind_T^G (\Phi_{T, \rho, \psi}) \cong j_{!*}^{\text{rss}}j^{\text{rss}*}\Ind_T^G (\Phi_{T, \rho, \psi}).
\]


\noindent Since $\Phi_{T, \rho, \psi}$ is $W$-equivariant (under $\iota_w$), and $\widetilde{G}^{\text{rss}} \to G^{\text{rss}}$ is a $W$-torsor, $W$ acts on $j^{\text{rss}*}\Ind_T^G (\Phi_{T, \rho, \psi})$. Thus, by functoriality of the intermediate extension, $W$ acts on $\Ind_T^G (\Phi_{T, \rho, \psi})$. 

At last we may define the $\gamma$-sheaf.

\begin{definition}
The $\gamma$-sheaf of $\rho$ on $G$, written $\Phi_{G, \rho, \psi}$, is defined to be the $W$-invariant direct factor of $\Ind_T^G (\Phi_{T, \rho, \psi})$. 
\end{definition} 


We will denote this simply as $\Phi_\rho$ when $G$ and $\psi$ are clear. Observe that 

\[
\Phi_\rho = j_{!*}^{\text{rss}}j^{\text{rss}*}\left(\Phi_\rho\right),
\]


\noindent so we may offer another definition of $\Phi_\rho$ as follows. If we restrict the sheaf $\Phi_{T, \rho, \psi}$ to $T^{\text{rss}}$, we see that it is a $W$-equivariant perverse sheaf on $T$. Thus it descends to a perverse local system on $T^{\text{rss}} / W$, which we shall call $\Phi_\lambda$. Then

\[
\Phi_\rho \cong j_{!*}^{\text{rss}}c^{\text{rss} *}\Phi_\lambda[d]\left(\frac{d}{2}\right).
\]


Thus we see that the $\gamma$-sheaf is, in essence, a Kloosterman sheaf on the maximal torus, extended by conjugacy-invariance to the regular semisimple locus $G^{\text{rss}}$ of $G$, and then extended (via the ``intermediate extension") to $G$. Thus ``generically" the associated function is a convolution of $\psi(t)$'s, like in the normalizing factors of Braverman-Kazhdan's normalized intertwining operators.

\subsection{Construction of the Fourier Transform}\label{BigFour}

Let the representation $\rho$ denote the adjoint representation of $M^{\vee}$ on $\mathfrak{u}_P^{\vee}$, and consider the corresponding $\gamma$-sheaf $\Phi_{M, \mathfrak{u}_P^{\vee}, \psi}$.\footnote{There is a philosophical reason to choose the representation $\mathfrak{u}_P^{\vee}$: in general, there seems to be a natural affinity

\begin{align*}
    G/[P,P] \leftrightsquigarrow (\mathfrak{u}_P^{\vee})^e\\
    G/U(P) \leftrightsquigarrow (\mathfrak{u}_P^{\vee}).
\end{align*}


\noindent where $e$ is the raising operator of an $\sl_2$ triple in $\mathfrak{m}^{\vee}$. See, for instance, the computation of the IC sheaves of the two corresponding Drinfeld compactifications \cite{BFGM, BG, LauEis, KuzSmall}. In \cite{BKIV}, the normalizing measure (on $M^{\ab}$) is given by a convolution of $\gamma$-functions associated to $(\mathfrak{u}_P^{\vee})^e$. The choice of $\mathfrak{u}_P^{\vee}$ is thus natural for our purposes.}

$\Phi_{M, \mathfrak{u}_P^{\vee}, \psi}$ is, of course, merely a sheaf on $M$; we \textit{really} want a sheaf on the Wang monoid $\overline{M}$. If we wish to preserve perversity, there is a canonical way to do this: letting $j: M \to \overline{M}$ be the open inclusion, we define:

\begin{equation}\label{defJ}
\mathfrak{J} : = j_{!*}(\Phi_{M, \mathfrak{u}_P^{\vee}, \psi}).
\end{equation}


Now, at last, we may state our conjecture. 


\textbf{Definition (Sheaf Version)}. Let $G$ be a split semi-simple group\footnote{The case of split reductive groups will be treated in section \ref{GenInt}.}, with parabolic subgroup $P$, Levi factor $M$, and maximal torus $T$. Let $\overline{M}$ be the Wang Monoid of $M$, $\mathfrak{S}$ be the Slipper pairing, and $\mathfrak{J}$ the $\gamma$-sheaf on $\overline{M}$ defined above. Consider the diagram:

\[
\begin{tikzcd}[column sep=15pt, row sep=15pt]
                  &  & \overline{G/U(P^{\text{op}})} \times \overline{G/U(P)} \arrow[llddd, "\pi_1"'] \arrow[rrddd, "\pi_2"]  \arrow[rr, "\mathfrak{S}"] & & \overline{M}\\
                  &  &  &  &    \\
                  &  &  &  &    \\
\overline{G/U(P^{\text{op}})} &  &  &  & \overline{G/U(P)}.
\end{tikzcd}
\]


\noindent We define the transform

\[
\mathcal{F}_{P^{\text{op}}, P} \, : \, D_c^{\flat}\left(\overline{G/U(P)}, \overline{\mathbb{Q}_l}\right) \to D_c^{\flat}\left(\overline{G/U(P^{\text{op}})}, \overline{\mathbb{Q}_l}\right)
\]


\noindent given by 

\begin{equation}\label{TransOp}
    \mathcal{F}(\mathscr{S}) = R\pi_{1\,!} \Bigl( \pi_2^*\left(\mathscr{S}\right) \otimes \left(\mathfrak{S}^{*}\left(\mathfrak{J}\right)[-m]\right) \Bigr)\left(\frac{n}{2}\right)[d],
\end{equation}


\noindent where $d = \dim \, G/U(P)$, $m = \dim \, M$ and $n = \dim \,G/U -\dim\, M = \dim \, U$.

We expect that there is a subcategory of $\mathcal{S} \subset D_c^{\flat}\left(\overline{G/U(P)}, \overline{\Q_\ell}\right)$ on which the transform $\mathcal{F}$  defines an involution:

\[
\mathcal{F}_{P,P^{\text{op}}}\circ\mathcal{F}_{P^{\text{op}}, P} (\mathscr{S}) \cong \mathscr{S}.
\]


\noindent We have also the ``function" version of the above:


\textbf{Definition (Function Version)}. Let $\langle x, y \rangle$, for $(x,y) \in (\overline{G/U} \times \overline{G/U^{\op}})(\F_q)$, denote $(-1)^nq^{-\frac{n}{2}}J(\mathfrak{S}(x,y))$. Then we let:

\[
\mathcal{F}_{P^{\op},P} \,:\, \text{Fun}(\overline{G/U(P)}(\F_q), \C) \to \text{Fun}(\overline{G/U(P^{\op})}(\F_q), \C),
\]


\noindent given by

\begin{equation}\label{FourSlip}
\mathcal{F}_{P^{\op},P}(f) (y) = \sum_{x \in \overline{G/U}} f(x) \langle x, y\rangle.
\end{equation}


\noindent This defines a $G \times M$-equivariant transform of function spaces. We expect that there is a natural class of $\C$-valued functions $\mathcal{S}$ on $\overline{G/U}(\mathbb{F}_q)$ for which $\mathcal{F}$ defines an involution:

\[
\mathcal{F}_{P,P^{\text{op}}}\circ\mathcal{F}_{P^{\text{op}}, P} (f) = f.
\]

\textbf{Question A}. What is the precise description of the category $\mathcal{S}$, and the corresponding constraint on functions? 

\vspace{3mm}

This is a generalization of the regular-monodromic condition of \cite{BravermanPolishchuk1998}. Likewise we might ask:

\vspace{3mm}

\textbf{Question B}. What are the Hecke-type (or braid-type) relations between $\mathcal{F}_{P,P^{\text{op}}}\circ\mathcal{F}_{P^{\text{op}}, P} (\mathscr{S})$ and $\mathscr{S}$ for an arbitrary $\mathscr{S} \in D^\flat(\overline{G/U(P)}, \overline{\Q}_\ell)$?

We will state a precise conjecture after we have introduced intertwiners between general $P$ and $P'$ in the next section.

\subsection{Review of Braverman-Kazhdan's Normalized Intertwiners for Borels}\label{BorelsRev} We would now like to construct normalized intertwiners for arbitrary pairs of parabolics -- not merely opposites. Our construction is motivated by the example of Borels; particularly adjacent Borels.
So it behooves us to review Braverman-Kazhdan's intertwiners for the case of Borels \cite{BKI}.

In this section, let $M= T$, the maximal torus of $G$. The parabolics $\mathcal{P}_T$ are the set of Borels containing $T$; they form a torsor under the Weyl group of $G$. Fix a Borel $B$ and let $B_w = n_w B n_w^{-1}$, where $n_w \in N(T)$ is a lift of $w$. This gives a fixed bijection between the $W$ and $\mathcal{P}_T$ whch will be convenient. We write simply $\mathcal{F}_{w',w}$ for the intertwiner:

\[
\mathcal{S}(\overline{G/U_w}, \C) \to \mathcal{S}(\overline{G/U_{w'}}, \C).
\]


\noindent We will work over a local field $F$; let $\psi: F \to \C$ a fixed additive character.

The idea (see also Kazhdan \cite{KazhdForm}), is to exploit the geometry of $G/U$ and $G/U_{s_\alpha}$ for $s_\alpha$ a simple reflection. In particular, it turns out that $G/U$ and $G/U_{s_{\alpha}}$ are naturally-dual rank-2 vector bundles (minus their 0-sections) over a common base-space. This linear structure in turn us allows us to write a transformation $\mathcal{F}_{B_{s_\alpha}, B}$ as a classical Fourier transform on each fiber. The involutivity property then follows straightforwardly from Fourier inversion for $\A^2$ of the kind we have already seen.

Let $B'= B_{s_\alpha}$, and $U' = U_{s_\alpha}$. Then $U \cap (U')^{\op}$ is a single 1-parameter subgroup associated to the root $\alpha$; we 
call 
this $u_\alpha: \mathbb{G}_a \to G$. We assume $G$ is simply connected, so that $u_{\alpha}$ and $u_{-\alpha}$ generate a subgroup isomorphic to $\SL_2$ in $G$. 

Let $Q$ denote the subgroup of $G$ generated by $U$ and $U'$. Note that $Q =[P_\alpha, P_\alpha]$, the commutator group of the minimal non-Borel parabolic associated to $\alpha$. Then we notice that:

\[
\begin{tikzcd}
G/U' \arrow[rd] &     & G/U \arrow[ld] \\
                  & G/Q &                   
\end{tikzcd}
\]


\noindent exhibits $G/U$ as a $Q/U \simeq SL_2/\mathcal{U} \simeq \A^2\setminus\{(0,0)\}$ fibration over $G/Q$. Likewise, $G/U'$ is a $Q/U' \simeq SL_2/\mathcal{U}^{\op} \simeq \A^2\setminus\{(0,0)\}$ fibration over $G/Q$. (Here $\mathcal{U}$, resp. $\mathcal{U}^{\op}$,  refers to the upper, resp. lower, unipotent subgroup of $\SL_2$.) In other words, both $G/U$ and $G/U'$ are rank two vector bundles over $G/Q$, with zero section removed.

Let $\widetilde{G/U}$ and $\widetilde{G/U'}$ denote the corresponding rank two vector bundles over $G/Q$; equivalently, these are the \textit{relative} affine closures of $G/U$ and $G/U'$ over $G/Q$ given by the relative spectrum $\Spec_{G/Q}\left( p_*(\mathcal{O}_{G/U})\right)$. There is a fiberwise duality between these two rank-2 vector bundles; in fact, there is a $G$-invariant form 

\begin{equation}\label{BravKazhPair}
\mathfrak{S}_{BK}: \widetilde{G/U} \times_{G/Q} \widetilde{G/U'} \to \A^1,
\end{equation}


\noindent which we call the \textbf{Braverman-Kazhdan Pairing}. The BK pairing reduces to the pairing we have already seen between $\overline{\SL_2/U}$ and $\overline{\SL_2/U^{\op}}$ in (\ref{sl2pairing}). Indeed, let $\phi: \SL_2 \to G$, where

\[
\phi\left(\begin{matrix}
    1 & x\\
    0 & 1
\end{matrix}\right) = u_\alpha(x); \,\,\,\,\, \phi\left(\begin{matrix}
    1 & 0\\
    x & 1
\end{matrix}\right) = u_{-\alpha}(x);\,\,\,\,\,
\phi\left(\begin{matrix}
    t & 0\\
    0 & t^{-1}
\end{matrix}\right) = \alpha^{\vee}(t).
\]


\noindent Then we observe that if $gU$ and $hU'$ both lie over the same point in $G/Q$, we have $U'h^{-1}gU$ is a well-defined element of $\phi(u_{-\alpha}) \backslash \phi(\SL_2) / \phi (u_{\alpha})$. Of course, this last double quotient is isomorphic (via $\phi^{-1}$) to 

\begin{align*}
    \mathcal{U}^{\op}\backslash SL_2 / \mathcal{U} &\cong \A^1
\end{align*}


\noindent via (\ref{doubleuinv}). As in the $\SL_2$ case, this pairing is clearly $G$-invariant since $U'h^{-1}gU = U'(g'h)^{-1}(g'g)U$.

Now, we define:

\[
C(\widetilde{G/U}(k), \C) \xrightarrow{\mathcal{F}_{B', B}}  C(\widetilde{G/U'}(k), \C)
\]


\noindent as the fiberwise Fourier transform with respect to $\psi$ and the above-defined vector-space duality. That is to say, we ``pull-push" functions along the upper roof of the Cartesian diagram:


\begin{equation}\label{pullpushadj}
\begin{tikzcd}
                              & \widetilde{G/U'} \times_{G/Q} \widetilde{G/U} \arrow[ld] \arrow[rd] &                                \\
\widetilde{G/U'} \arrow[rd] &                                                                           & \widetilde{G/U} \arrow[ld] \\
                              & G/Q                                                                       &                               
\end{tikzcd}
\end{equation}


\noindent with respect to the kernel $\psi\left(\mathfrak{S}_{BK}(-,-)\right)$.

In formulae, let $f \in C(\widetilde{G/U}(k), \C)$. We pull $f$ back to a right $U$-invariant function on $G$, which we, somewhat abusively, also call $f$. Then, for all $g \in G$, and $\beta, \delta \in k$, we have:

\[
    \left(\mathcal{F}_{B', B} (f)\right) \left[ g \phi\left(\begin{matrix}
        * & \beta\\
        * & \delta
    \end{matrix}\right) \right]=  \int_{\A^2} f\left(g \phi\left(\begin{matrix}
        \alpha & *\\
        \gamma & *
    \end{matrix}\right) \right) \psi\left( \alpha \delta - \beta \gamma \right) d\alpha d\gamma,
\]


\noindent where $*$ means arbitrary entries such that $\left(\begin{matrix}
        * & \beta\\
        * & \delta
    \end{matrix}\right), \left(\begin{matrix}
        \alpha & *\\
        \gamma & *
    \end{matrix}\right) \in \SL_2(k)$. 
(The output is independent of the choice of $*$'s because $f$ is right $U(B)\supset u_\alpha$-invariant, and $\mathcal{F}_{B', B} (f)$ is right $U(B') \supset u_{-\alpha}$-invariant.) 

\subsection{Finite Field Subtleties} The geometry involved in these Fourier transforms becomes somewhat more transparent if we recall the elegant $G\times T$-equivariant stratification

\begin{equation}\label{stratGU}
\overline{G/U(B)} = \bigsqcup_{P \supset B} G/[P,P]
\end{equation}


\noindent for all parabolic subgroups $P \supseteq B$. The smooth locus consists precisely of those $G/[P,P]$ coming from the subminimal parabolics containing $B$; i.e., if $\theta$ represents a subset of simple roots (with respect to $B$) and $P_\theta$ represents the corresponding parabolic, then $\overline{G/U(B)} = \bigsqcup_{\theta} G/[P_\theta,P_\theta]$ while 

\begin{equation}\label{stratGUsm}
    \overline{G/U(B)}^{\textrm{sm}} = \bigsqcup_{\theta: \,|\theta|\le 1} G/[P_\theta,P_\theta].
\end{equation}


\noindent (See \cite{GannonGK}, Proposition 5.1.) This in turn corresponds to the union of all of the ``0 sections" that we append when taking the relative affine closure of $G/U$ with respect to $G/[P_\alpha, P_\alpha]$ (for $\alpha$ a simple root of $G$).

Now, if we attempt to repeat Braverman-Kazhdan's construction in the context of finite fields, we quickly hit an obstruction. We may perfectly well pull-push functions on rational points along (\ref{pullpushadj}); this will define an involutive Fourier transform $\mathcal{F}_{s_\alpha,e}$. But we can proceed no further in composing this transform by \textit{another} such 
transform\footnote{
Observe that while $B$ is adjacent to $B_{s_\alpha}$ for a simple reflection $s_\alpha$,  it is no longer true that $B_{s_\alpha}$ is adjacent to $B_{s_{\beta}s_{\alpha}}$ for $s_\beta$ another simple reflection. Rather, we see that we must conjugate $s_\beta$ by $s_\alpha$ to produce simple reflection with respect to the Borel $B_{s_\alpha}$. Hence we find that $B_{(s_\alpha s_\beta s_\alpha) s_\alpha} = B_{s_\alpha s_\beta}$ is the desired adjacent Borel. Notice that the order of $\alpha$ and $\beta$ is flipped from what one might naively expect.} 
$\mathcal{F}_{s_\alpha s_\beta,s_\alpha}$: applying the transform $\mathcal{F}_{s_\alpha,e}$ results in a function on $\widetilde{G/U'}(\F_q)$, where the relative affine closure has been taken with respect to $G/[P_\alpha,P_\alpha]$; the domain of $\mathcal{F}_{s_\alpha s_\beta,s_\alpha}$ consists of the set of functions on a \textit{different} relative affine closure of $G/U'$. Of course, over local fields, these discrepancies are irrelevant because the difference between $\widetilde{G/U}$ and $G/U$ is of measure 0, so we may restrict to functions on $G/U$ everywhere. But, as we have already mentioned, Zariski-closed subsets of varieties over finite fields have positive measure. We cannot simply ignore them.

There are two ways we might attempt to circumvent this problem. One approach, implemented in Sections \ref{RFC} and \ref{Sp4Func} below for the case of $\SL_3$ and $\Sp_4$, is to restrict our space of functions to those functions $f$ such that the support of $\mathcal{F}_{s_\alpha,e}(f)$ is entirely in $G/U'(\F_q) \subset \widetilde{G/U'}(\F_q)$. Another approach is to try to extend these ``adjacent" Fourier transforms to all functions on the absolute affine closure $\overline{G/U}$ from the outset. However, even the BK-pairing fails to extend to absolute affine closures (see the paragraph before Section \ref{CompRes} below for the specific case of $\SL_3$), so there are geometric obstacles to accomplishing this.

\subsection{Partial Intertwiners For Arbitrary Pairs of Parabolics}\label{GenInt} Recall that $\mathcal{P}_M$ denotes the set of parabolics containing the fixed Levi $M$. We now provide the conjectural description of normalized intertwiners for all pairs of parabolics $P, P' \in \mathcal{P}$, including non-opposites.

Let $Q$ denote the subgroup of $G$ generated by $U(P)$ and $U(P')$. Observe that $R_u(Q) = R_u(R) = U(P) \cap U(P')$. Consider the quotient $L:=Q/R_u(Q)$. This is a reductive (in fact semisimple) group; by the Levi splitting, we may write $L \subset Q \subset G$. The inclusion of $L$ into $Q$ (and so $G$) is canonical if we insist, as we shall, that $L$'s maximal torus be a subtorus of the maximal torus of $G$.

Now we let $P_L^{+} := P \cap L$, $P_L^{-} := P' \cap L$; likewise, we let $U_L^{+}: =U(P) \cap L$ and $U_L^{-}: =U(P') \cap L$. As the notation suggests, $P_L^{+}$ and $P_L^{-}$ are opposite parabolic subgroups of $L$, with corresponding Levi factor $M \cap L$. Now we consider the Wang monoid of $M\cap L$ in $L$:

\[
\overline{M\cap L} : = \Spec\, k[L]^{U_L^{-} \times U_L^{+}}.
\]


\noindent We let the sheaf $\mathfrak{J}_{\overline{M\cap L}}$ on $\overline{M\cap L}$ be defined as in (\ref{defJ}); i.e., the intermediate extension of the BK-sheaf $\Phi_{M\cap L,\, \mathfrak{u}_L^{+\,\vee}, \,\psi}$ on $M\cap L$ (with $\mathfrak{u}_L^{+}$  the adjoint representation of $(M\cap L)^{\vee}$ on $\mathfrak{u}_L^{+\,\vee}$), to $\overline{M\cap L}$.

Now, we have a map -- a generalized pair of Slipper's\footnote{One might say, arbitrary footwear.}:

\begin{equation}\label{generalSlip}
\mathfrak{S}: G/U(P') \times_{G/Q} G/U(P) \to \overline{M \cap L},
\end{equation}
\[
(gU', hU) \mapsto U_L^{-}g^{-1}hU_L^{+}.
\]


\noindent Indeed: observe that the fiber of $G/U(P)$ over $G/Q$ is:

\begin{equation}\label{fiberman}
Q/U(P) \simeq \left[L\cdot\left(U(P)\cap U(P')\right)\right]/U(P) \simeq \left[L\cdot U(P)\right]/U(P) \simeq L/(U(P) \cap L) \simeq L/U_L^{+},
\end{equation}


\noindent and similarly, the fiber of $G/U(P')$ over $G/Q$ is $L/U_L^{-}$. So, by $G$-invariance, (\ref{generalSlip}) reduces to the pairing:

\begin{equation}\label{LRel}
L/U_L^{-} \times L/U_L^{+} \to \overline{M \cap L},
\end{equation}


\noindent which is simply the restriction of the usual Slipper pairing $\overline{L/U_L^{-}} \times \overline{L/U_L^{+}}\to \overline{M \cap L}$ to $L/U_L^{-} \times L/U_L^{+}$. 

Now, each variety in the fiber product on the LHS of (\ref{generalSlip}) -- the parspherical spaces $G/U(P)$, $G/U(P')$, and the space $G/Q$\footnote{In the case of $P$ and $P'$ two Borels, $G/Q$ is actually a BK space $G/[R,R]$ where $R$ is the parabolic generated by $B$ and $B'$.}, are all quasi-affine. However, in general we cannot extend the generalized Slipper pairing $\mathfrak{S}$ to

\[
\overline{G/U(P)} \times_{\overline{G/Q}} \overline{G/U(P')}; 
\]


\noindent we already see this phenomenon for adjacent Borels in $\SL_3$ (see the paragraph before Section \ref{CompRes} below). Nor, indeed, can we remove the necessity for the Wang Monoid on the RHS of (\ref{generalSlip}); even in (\ref{LRel}) we see that two elements of $L$ might not differ by an element of the big cell $U_L^{-} (M\cap L) U_L^{+}$. However, we can extend the pairing to relative affine closures, as we have seen with the BK pairing:

\begin{equation}\label{SlipPairPart}
\mathfrak{S}: \widetilde{G/U(P)} \times_{G/Q} \widetilde{G/U(P')} \rightarrow \overline{M \cap L},
\end{equation}


\noindent where $\widetilde{G/U(P)}$ denotes the relative affine closure of $G/U(P)$ with respect to $G/Q$. Observe that, fiberwise, this gives the usual Slipper pairing, $\overline{L/U_L^{-}} \times \overline{L/U_L^{+}}\to \overline{M \cap L}$.

Now we may describe the intertwiner $\mathcal{F}_{P',P}$. We let $G$ be any split reductive group over $\F_q$.\footnote{In formulating the Fourier transform before, recall that we had restricted to $G$ semisimple.} Consider the diagram

\[
\begin{tikzcd}
                    & \widetilde{G/U(P')} \times_{G/Q} \widetilde{G/U(P)} \arrow[r, "\mathfrak{S}"] \arrow[ld, "\pi_1"'] \arrow[rd, "\pi_2"] & \overline{M\cap L} \\
\widetilde{G/U(P')} &                                                                                                                        & \widetilde{G/U(P)}
\end{tikzcd}
\]


\noindent We define the transform

\[
\mathcal{F}_{P', P} \, : \, D_c^{\flat}\left(\widetilde{G/U(P)}, \overline{\mathbb{Q}_l}\right) \to D_c^{\flat}\left(\widetilde{G/U(P')}, \overline{\mathbb{Q}_l}\right)
\]


\noindent given by 

\begin{equation}\label{PP'trnas}
    \mathcal{F}_{P',P}(\mathscr{S}) = R\pi_{1\,!} \Bigl( \pi_2^*\left(\mathscr{S}\right) \otimes \left(\mathfrak{S}^{*}\left(\mathfrak{J}_{\overline{M\cap L}}\right)[-m]\right) \Bigr)\left(\frac{n}{2}\right)[d],
\end{equation}


\noindent where $m = \dim\, M \cap L$, $d = \textrm{codim}\, \pi_2 = \dim\, G/U - \dim \,G/Q = \dim \,Q - \dim \, U = \dim\, L - \dim \, U_{L}^{+}$ (as per \ref{fiberman}) and $n = \dim\, G/U - \dim \,G/Q - \dim\, M \cap L$.

\vspace{3mm}

\begin{conj}\textbf{(Parabolic Kazhdan-Laumon.)}\label{MainConj} The restriction of the kernel in (\ref{PP'trnas}) to $G/U(P') \times_{G/Q}G/U(P)$, as $P$ ranges over all parabolics of $G$ containing $M$ (and with Levi quotient isomorphic to $M$), defines a Kazhdan-Laumon gluing datum for $G/U(P)$.
\end{conj}

\vspace{3mm}

This would be particularly noteworthy in those cases where $G/U(P)$ and $G/U(P')$ are not isomorphic as varieties.\footnote{In the case of Borels and basic affine space, the varieties $G/U(B)$ are isomorphic for all $B$, since all Borels are conjugate. It does not appear that $G/U(P)$ and $G/U(P')$ need to be isomorphic varieties for general associate $P$, the first example being the various 2+1+1 parabolics of $\SL_4$.}

Now, we expect that there is a natural category of sheaves on $\mathcal{S}$ for which $\mathcal{F}_{P',P}$ is involutive\footnote{I.e., satisfies $\mathcal{F}_{P,P'}\circ \mathcal{F}_{P',P} = \Id$.}; moreover, if we restrict to the subcategory of $\mathcal{S}$ of sheaves $\mathscr{S}$ such that both $\mathscr{S}$ and $\mathcal{F}_{P',P}(\mathscr{S})$ have support contained in $G/U$, we expect that $\mathcal{F}_{P',P}$ will define normalized intertwiners.

\textbf{Example 1}. Let us consider the most trivial case: $P = P'$. We see that $Q = U(P)$, so $L = Q/R_u(Q)$ is the trivial group $\{*\} =\Spec\,\F_q$. Thus the Wang monoid is also $\{*\}$, and the sheaf $\mathfrak{J}$ is the constant sheaf $\overline{\Q_\ell}$. We have the correspondence:

\[
\begin{tikzcd}
    & G/U\times_{G/U} G/U \arrow[r, "\mathfrak{S}"] \arrow[ld, "\pi_1"'] \arrow[rd, "\pi_2"] & \{*\} \\
G/U &                                                                                        & G/U  
\end{tikzcd}
\]


\noindent which shows that $\mathcal{F}_{P,P}$ is pull-push with respect to the kernel given buy the constant sheaf $\overline{\Q_\ell}$. Thus $\mathcal{F}_{P,P}= \Id$, as demanded by (\ref{Id1}).

\textbf{Example 2}. Consider $G$ an arbitrary split, semisimple simply-connected group over $\F_q$, with $B$ and $B'$ adjacent Borels. Then $Q$, the subgroup generated by $U$ and $U'$, is $[P_\alpha, P_\alpha]$, where $P_\alpha$ is the sub-minimal parabolic generated by $B$ and $B'$. In this case $L \cong \SL_2$, $L \cap M =\G_m$, and $\overline{L\cap M} = \A^1$, exactly as in (\ref{doubleuinv}). We find that (\ref{generalSlip}) is simply the Braverman-Kazhdan pairing of (\ref{BravKazhPair}). Meanwhile, the sheaf $\mathfrak{J}$ on $\overline{L\cap M}$ is simply the Artin-Schreier sheaf $\mathcal{L}_{\psi}[1]\left(\frac{1}{2}\right)$, as we have already computed from the case of $\SL_2$. Thus, generically, we see that our proposed Fourier transform is the local system $\mathfrak{S}_{BK}^*(\mathcal{L}_\psi)(1)$, which agrees with the 2-dimensional symplectic linear Fourier transform of Braverman and Kazhdan.

\noindent Examples 1 and 2 consider the case of $P$ and $P'$ close together; for the opposite extreme.

\textbf{Example 3}. Say $P' = P^{\op}$. Then $Q = G^{\textrm{der}}$ (so in particular, if $G$ is semi-simple, $Q=G$). Then $G/Q = G^{\ab}$, which is trivial for $G$ semisimple. Therefore, in this case, we see that we recover precisely the ``opposite" Fourier transform from (\ref{BigFour}). 

Thus we see that the general (partially defined) pairing (\ref{SlipPairPart}) generalizes both the pair of Slipper's and the Braverman-Kazhdan pairing.

\section{The \texorpdfstring{$(n-1)+1$}{(n-1)+1} Parabolic of \texorpdfstring{$\SL_n$}{SLn}}\label{Mirabolic}

Let $G = \SL_n$ and consider the $(n-1) + 1$ block parabolic:

\[
P = \left\{\left(\begin{matrix}
    a_{1,1} & a_{1,2} & \cdots & a_{1, n-1} & a_{1,n}\\
    a_{2,1} & a_{2,2} & \cdots & a_{2, n-1} & a_{2,n}\\
    \vdots & \vdots & \ddots & \vdots & \vdots \\
    a_{n-1, 1} & a_{n-1, 2} & \dots & a_{n-1, n-1} & a_{n-1, n}\\
    0 & 0 & \dots & 0 & a_{n,n}\\
\end{matrix}\right)\right\}.
\]


\noindent The unipotent radical is isomorphic to the vector space $k^{n-1}$:

\[
U(P) = \left\{\left(\begin{matrix}
    1 & 0 & \cdots & 0 & a_{1,n}\\
    0 & 1 & \cdots & 0 & a_{2,n}\\
    \vdots & \vdots & \ddots & \vdots & \vdots \\
    0 & 0 & \dots & 1 & a_{n-1, n}\\
    0 & 0 & \dots & 0 & 1\\
\end{matrix}\right)\right\},
\]


\noindent while the Levi $M \cong \GL_{n-1}$ is given by:

\[
P = \left\{\left(\begin{matrix}
    a_{1,1} & a_{1,2} & \cdots & a_{1, n-1} & 0\\
    a_{2,1} & a_{2,2} & \cdots & a_{2, n-1} & 0\\
    \vdots & \vdots & \ddots & \vdots & \vdots \\
    a_{n-1, 1} & a_{n-1, 2} & \dots & a_{n-1, n-1} & 0\\
    0 & 0 & \dots & 0 & D^{-1}\\
\end{matrix}\right)\right\}.
\]


\noindent Here $D = \det[a_{i,j}]_{1 \le i,j\le n-1}$. Since $D$ is determined by the upper left $(n-1) \times (n-1)$-block, we will often write $M$ as just the collection of matrices $[a_{i,j}]_{1 \le i,j\le n-1}$.

We see that the quotient $G/U(P)$ isomorphic to the locus of matrices in $\Mat_{n,n-1}$ of full rank (i.e., $\text{rank} = n-1$):

\begin{align*}
\left(\begin{matrix}
    a_{1,1} & a_{1,2} & \cdots & a_{1, n-1} & a_{1,n}\\
    a_{2,1} & a_{2,2} & \cdots & a_{2, n-1} & a_{2,n}\\
    \vdots & \vdots & \ddots & \vdots & \vdots \\
    a_{n-1, 1} & a_{n-1, 2} & \dots & a_{n-1, n-1} & a_{n-1, n}\\
    a_{n, 1} & a_{n, 2} & \dots & a_{n, n-1} & a_{n,n}\\
\end{matrix}\right) \left(\begin{matrix}
    1 & 0 & \cdots & 0 & *\\
    0 & 1 & \cdots & 0 & *\\
    \vdots & \vdots & \ddots & \vdots & \vdots \\
    0 & 0 & \dots & 1 & *\\
    0 & 0 & \dots & 0 & 1\\
\end{matrix}\right) \\
\mapsto \left(\begin{matrix}
    a_{1,1} & a_{1,2} & \cdots & a_{1, n-1} \\
    a_{2,1} & a_{2,2} & \cdots & a_{2, n-1}\\
    \vdots & \vdots & \ddots & \vdots & \\
    a_{n-1, 1} & a_{n-1, 2} & \dots & a_{n-1, n-1}\\
    a_{n, 1} & a_{n, 2} & \dots & a_{n, n-1} &\\
\end{matrix}\right).
\end{align*}


\noindent Indeed, the image of this $\GL_n$ under this map is $\Mat_{n,n-1}^{\text{rk} = n-1} \subset \Mat_{n,n-1}$. On the other hand the quotient variety $G/U(P)$ has dimension $(n^2-1)-(n-1) = n(n-1) = \dim\, \Mat_{n,n-1}$. Moreover, we can verify that the map is separable and bijective on points, proving the isomorphism of varieties. The degenerate locus in $\Mat_{n,n-1}$ has codimension 2 if $n=2$, and codimension $n-1$ if $n \ge 3$. In all these cases the codimension is $\ge 2$, whence, by Hartog's Lemma we see that the affine closure $\overline{G/U(P)} \cong \Mat_{n,n-1}$. Note that this is smooth (and in fact an affine space!). This is highly atypical -- part of the the ``miraculous" nature of this particular parabolic subgroup.

Next we describe the space $\overline{G/U(P^{\op})}$, the Wang Monoid, and the Slipper pairing. $\overline{G/U^{\op}}$ can be naturally identified $\Mat_{n-1, n}$ as follows: consider $\left(gU^{\op}\right)^{-1} = U^{\op} g^{-1}$. Then we have\footnote{We write $a_{ij}^{-1}$ for $([a]^{-1})_{ij}$; we do not mean $1/a_{ij}$.}:

\begin{align*}
 \left(\begin{matrix}
    1 & 0 & \cdots & 0 & 0\\
    0 & 1 & \cdots & 0 & 0\\
    \vdots & \vdots & \ddots & \vdots & \vdots \\
    0 & 0 & \dots & 1 & 0\\
    * & * & \dots & * & 1\\
\end{matrix}\right)\left(\begin{matrix}
    a_{1,1}^{-1} & a_{1,2}^{-1} & \cdots & a_{1, n-1}^{-1} & a_{1,n}^{-1}\\
    a_{2,1}^{-1} & a_{2,2}^{-1} & \cdots & a_{2, n-1}^{-1} & a_{2,n}^{-1}\\
    \vdots & \vdots & \ddots & \vdots & \vdots \\
    a_{n-1, 1}^{-1} & a_{n-1, 2}^{-1} & \dots & a_{n-1, n-1}^{-1} & a_{n-1, n}^{-1}\\
    a_{n, 1}^{-1} & a_{n, 2}^{-1} & \dots & a_{n, n-1}^{-1} & a_{n,n}^{-1}\\
\end{matrix}\right) \\
\mapsto \left(\begin{matrix}
    a_{1,1}^{-1} & a_{1,2}^{-1} & \cdots & a_{1, n-1}^{-1} & a_{1,n}^{-1}\\
    a_{2,1}^{-1} & a_{2,2}^{-1} & \cdots & a_{2, n-1}^{-1} & a_{2,n}^{-1}\\
    \vdots & \vdots & \ddots & \vdots & \\
    a_{n-1, 1}^{-1} & a_{n-1, 2}^{-1} & \dots & a_{n-1, n-1}^{-1} & a_{n-1,n}^{-1}\\
\end{matrix}\right).
\end{align*}


\noindent Now, the Wang monoid is by definition $\Spec \, k[G]^{U(P) \times U(P^{\op})}$. We see that the double-unipotent-invariants are $a_{ij}$, for $1 \le i, j \le n-1$:

\begin{align*}
 \left(\begin{matrix}
    1 & 0 & \cdots & 0 & 0\\
    0 & 1 & \cdots & 0 & 0\\
    \vdots & \vdots & \ddots & \vdots & \vdots \\
    0 & 0 & \dots & 1 & 0\\
    * & * & \dots & * & 1\\
\end{matrix}\right)\left(\begin{matrix}
    a_{1,1} & a_{1,2} & \cdots & a_{1, n-1} & a_{1,n}\\
    a_{2,1} & a_{2,2} & \cdots & a_{2, n-1} & a_{2,n}\\
    \vdots & \vdots & \ddots & \vdots & \vdots \\
    a_{n-1, 1} & a_{n-1, 2} & \dots & a_{n-1, n-1} & a_{n-1, n}\\
    a_{n, 1} & a_{n, 2} & \dots & a_{n, n-1} & a_{n,n}\\
\end{matrix}\right) \left(\begin{matrix}
    1 & 0 & \cdots & 0 & *\\
    0 & 1 & \cdots & 0 & *\\
    \vdots & \vdots & \ddots & \vdots & \vdots \\
    0 & 0 & \dots & 1 & *\\
    0 & 0 & \dots & 0 & 1\\
\end{matrix}\right)\\
=\left(\begin{matrix}
    a_{1,1} & a_{1,2} & \cdots & a_{1, n-1} & *\\
    a_{2,1} & a_{2,2} & \cdots & a_{2, n-1} & *\\
    \vdots & \vdots & \ddots & \vdots & \vdots \\
    a_{n-1, 1} & a_{n-1, 2} & \dots & a_{n-1, n-1} & *\\
    * & * & \dots & * & *\\
\end{matrix}\right)
\end{align*}


\noindent Thus the Wang Monoid is $\Mat_{n-1, n-1}$ (and observe that the Levi, isomorphic to $\GL_{n-1}$, is indeed an open subset).

Under these identifications the Slipper pairing $\mathfrak{S}(gU(P), hU(P^{\op})) := U(P^{\op})h^{-1}gU(P)$ is simply given by the multiplication of matrices:

\begin{align}\label{SlipPairSln}
    \Mat_{n, n-1} \times \Mat_{n-1, n} &\to \Mat_{n-1,n-1}\\
    \nonumber (M, N) &\mapsto NM.
\end{align}


Now, we will compute the function $J$. To do this we must calculate the weight decomposition of the adjoint representation of $M^{\vee}$ on $\mathfrak{u}_P^{\vee}$. Recall that the Langlands dual of $\SL_n$ is $\PGL_n$, and that the dual of $\GL_{n-1}$ is $\GL_{n-1}$. Thus the weights of $\PGL_n$ are given by 

\[
\left\{\sum_i a_i\pi_i : \sum a_i = 0 \right\}
\]


\noindent where $\pi_i: \text{diag}(t_1, \ldots, t_n) \mapsto t_i$. We see that the representation $\mathfrak{u}_P$ has weights $\{\pi_i - \pi_n\}_i$ for $i = 1, \ldots n-1$. These correspond to the cocharacters $e_i - e_n : \mathbb{G}_m \to T_M = T_G$, where $e_i: t \to \text{diag}(1, \ldots ,t, \ldots ,1) \subset \GL_{n-1}$, with $t$ in the $i$th position ($1 \le i \le n-1$). 

Restricting to the upper $(n-1) \times (n-1)$ block, we see that the coweights in $\GL_{n-1}$ are simply $e_i$, for $i = 1, \ldots, n-1$. Thus we have the diagram:

\[
\begin{tikzcd}
                      & \mathbb{G}_m^{n-1} \arrow[ld] \arrow[rd]                        &                                     \\
\mathbb{A}^1          & {(t_1, \ldots t_{n-1})} \arrow[ld, maps to] \arrow[rd, maps to] & T_{\text{GL}_{n-1}}                 \\
t_1 + \cdots +t_{n-1} &                                                                 & {\text{diag}(t_1, \ldots, t_{n-1})}
\end{tikzcd}
\]


\noindent Pushing and pulling the character $\psi$ from $\mathbb{A}^1$ to $T$, we see that for all $t \in T$, the $\gamma$-sheaf for the torus $\Phi_{T, \mathfrak{u}_P^{\vee}, \psi}$ has function $q^{\frac{1-n}{2}}(-1)^{n-1}\psi(\text{tr}(t))$, where $\tr: T \to \A^1$ is the standard trace.

Consider the local system $\mathcal{E}$ on $\Mat_{n-1}$ obtained by pulling back the Artin-Schreier sheaf $\mathcal{L}_\psi$ along the map $\tr: \Mat_{n-1} \to \A^1$. Since this matches the $\gamma$-sheaf $\Phi_{T, \mathfrak{u}_P^{\vee}, \psi}$ when restricted to $T$, we find that the $\gamma$-sheaf $\Phi_{M, \mathfrak{u}_P^{\vee}, \psi}$ agrees with $\mathcal{E}$ when we restrict both sheaves to the rss locus of $\GL_{n-1}$. Thus, by the  ``principal of perverse continuation" \cite{NgoPCMI}, we see that $\mathfrak{J}$, the intermediate extension of $\Phi_{M, \mathfrak{u}_P^{\vee}, \psi}$ to $\Mat_{n-1}$, is isomorphic to $\mathcal{E}$ (up to dimension shift and Tate twist). Thus the function $J$ is simply

\begin{equation}\label{gammasln}
J(m) = q^{-\frac{(n-1)^2}{2}}(-1)^{(n-1)^2}\psi(\tr(m)).
\end{equation}


\noindent Now, the Slipper pairing $\mathfrak{S}$ given  by (\ref{SlipPairSln}) is not smooth; however, the  pullback of a local system is a local system, which, since $\overline{G/U} \times \overline{G/U^{\op}}$ is smooth in this case, is perverse (up to a dimension shift). We see that our function is given by $\langle M, N \rangle = q^{-\frac{n^2 - n}{2}}\psi(tr(M))$. Therefore our Fourier transform is given by:

\[
\mathcal{F}_{P^{\op},P}(f)(N) = q^{-\frac{n^2 - n}{2}}\sum_{M \in \Mat_{n,n-1}} f(M) \psi(\tr(NM)).
\]


\noindent But this is simply a standard Fourier transform on vector spaces. Because of this, we see that involutivity holds as a consequence of standard Fourier inversion for vector spaces: $\mathcal{F}_{P,P^{\op}}\mathcal{F}_{P^{\op},P}(f) = f$. Moreover, this readily upgrades to sheaves, as a result of the Fourier-Deligne inversion formula.

We see that in this ``miraculous" case, the Fourier transform is involutive on all sheaves (and functions) on the affine closure. This appears to be related to the smoothness of $\overline{\SL_n/U}$. In the next section we will consider the simplest singular case, where restriction becomes necessary.

\section{The Kloosterman Fourier Transform on Affine Quadric Cones}

We now describe an involutive Fourier transform for a class of functions on certain affine quadric cones. This will prove involutivity for (a natural subspace of) functions on two more examples of paraspherical space: opposite Borels for $\SL_3$ and opposite Siegel parabolics for $\Sp_4$.

Let $V \cong k^d$ be a vector space over $\mathbb{F}_q$ of dimension $d$. We let $\psi: \mathbb{F}_q \to \C$ be a nontrivial additive character, as always. We let

\[
X = \{(v, v^{\vee}) \in V \times V^* : \langle v, v^\vee\rangle = 0\}.
\]


\noindent Observe that $X$ is an affine quadric cone in $k^{2d}$, with an isolated conical singularity at the origin. There is a natural scalar action of $\mathbb{G}_m$ on this variety: $\lambda \cdot(v, w) = (\lambda v, \lambda w)$ for $\lambda \in \mathbb{F}_q^{\times}$.

For $a \in \mathbb{F}_q$ we let 

\[
\Kl(a) := \sum_{t \in \mathbb{F}_q^{\times}} \psi\left(\frac{a}{t} + t\right),
\]


\noindent be the standard Kloosterman function on $\mathbb{F}_q$. Now consider the composite map 

\[
K: X \times X \to \mathbb{A}^1 \to \C
\]
\[
\left((u, u^{\vee}),(v, v^{\vee})\right) \mapsto \langle u, v^{\vee}\rangle + \langle v, u^{\vee}\rangle \mapsto \Kl\left(\langle u, v^{\vee}\rangle + \langle v, u^{\vee}\rangle\right). 
\]


\noindent Let 

\begin{equation}\label{quadshw}
\mathcal{S}(X(\mathbb{F}_q),\C) =  \left\{f \in \text{Fun}(X(\mathbb{F}_q),\C) : \sum_{\lambda \in \mathbb{F}_q^{\times}} f(\lambda x) = 0\right\}.
\end{equation}


\noindent We call this the space of ``special" functions. We define on $\mathcal{S}$ the following transform:

\begin{equation}\label{FourKl}
\mathcal{F}(f)(y) = \sum_{x \in X(\mathbb{F}_q)} f(x)K(x,y).
\end{equation}


Before we state our theorem, let us observe two things about the space $\mathcal{S}$. Firstly, we may see immediately that $f(0) = 0$. Thus $f \in \mathcal{S}$ is solely supported on the smooth locus of the variety $X$. Secondly, we observe that $\mathcal{F}$ preserves the space $\S$. Indeed:

\[
\sum_{\lambda\in\F_q^{\times}}\mathcal{F}(f)(\lambda y) = \sum_{\lambda\in\F_q^{\times}}\sum_{x \in X(\mathbb{F}_q)} f(x)K(x,\lambda y).
\]


\noindent But letting $x = (u, u^{\vee})$ and $y =(v, v^{\vee})$, we see that  $K(x,\lambda y) = \Kl(\langle u,\lambda v^{\vee}\rangle + \langle\lambda v, u^{\vee}\rangle) = \Kl(\langle \lambda u, v^{\vee}\rangle + \langle v, 
 \lambda u^{\vee}\rangle) = K(\lambda x, y)$ by bilinearity. Thus we obtain:

 \begin{align*}
\sum_{\lambda\in\F_q^{\times}}\sum_{x \in X(\mathbb{F}_q)} f(x)K(x,\lambda y) =& \sum_{\lambda\in\F_q^{\times}}\sum_{x \in X(\mathbb{F}_q)} f(x)K(\lambda x,y)\\
=& \sum_{x \in X(\mathbb{F}_q)} \sum_{\lambda\in\F_q^{\times}} f(\lambda^{-1}x)K( x,y)\\
=& \,0.
\end{align*}


\noindent (As an aside will also note that $\sum_{\lambda \in \F_q^{\times}}\sum \Kl(x, \lambda y) = 1$, unless $\langle u, v^{\vee}\rangle + \langle v, 
 u^{\vee}\rangle = 0$, in which case it is $1-q$.)

 We are now ready to state our involutivity theorem:

\begin{thm}\label{KlInv}
The transform $\mathcal{F}$ satisfies

 \[
\mathcal{F}^2(f) = q^{2d}f 
 \]


\noindent for all $f \in \mathcal{S}(X(\mathbb{F}_q),\C)$.
\end{thm}

\begin{proof}
    Observe that the space $\mathcal{S}$ is generated by functions of the form $\delta_x - \delta_{\lambda x}$, where $x = (u, u^{\vee}) \in X$ and $\delta_x$ is the indicator function of $x \in X(\mathbb{F}_q)$. The theorem therefore boils down to verifying the following identity:

    \begin{equation}\label{KSum}
        \sum_{y \in X} \left[K(x,y)- K(\lambda x, y)\right]K(y,z) = \begin{cases}
            \,q^{2d} \text{  if  }z=x\\
            -q^{2d} \text{  if  }z=\lambda x\\
            0 \text{  otherwise  }
        \end{cases}
    \end{equation}


\noindent We will begin by considering the double Fourier transform of the Dirac mass at $x$:

\begin{equation}\label{FourDirac}
     \mathcal{F}^2(\delta_x)(z) = \sum_{y \in X(\F_q)} K(x,y)K(y,z).
\end{equation}


\noindent As above, let $x = (u, u^{\vee})$, and $z = (w, w^{\vee})$ (where, of course, $\langle u, u^{\vee} \rangle =\langle w, w^{\vee} \rangle =0$). Developing the sum on the left hand side of (\ref{FourDirac}) we obtain:

\begin{equation}
\sum_{(v,v^{\vee}) \in X(\F_q)}\left(\sum_{t \in \F_q^{\times}} \psi\left[t\left( \langle u, v^{\vee}\rangle  + \langle v, u^{\vee} \rangle\right)  + t^{-1}\right]\right)
\cdot\left(\sum_{s \in \F_q^{\times}} \psi\left[s\left( \langle v, w^{\vee}\rangle  + \langle w, v^{\vee} \rangle\right) + s^{-1}\right] \right)
\end{equation}


\noindent or, swapping the order of the summation:

\[
\sum_{t,s \in \F_{q}^{\times}}\sum_{(v,v^{\vee}) \in X(\F_q)} \psi\left[t\left( \langle u, v^{\vee}\rangle  + \langle v, u^{\vee} \rangle\right)\right]\cdot\psi(t^{-1})
\cdot \psi\left[s\left( \langle v, w^{\vee}\rangle  + \langle w, v^{\vee} \rangle\right)\right]\cdot\psi(s^{-1}).
\]


Now we recall that $X(\F_q)$ is invariant under scaling: thus we may apply the substitution\footnote{
Applying a substitution that eliminates the intractable Kloosterman sum terms when, as in our situation, we have a product of \textit{two} such Kloosterman sums, bears a striking resemblance to the squaring (followed by polar substitution) technique used to resolve the Gaussian integral.
} 
$(v,v^{\vee}) \mapsto (t^{-1}s^{-1}v, t^{-1}s^{-1}v^{\vee})$ to the point $(v, v^{\vee}) \in X(\F_q)$. Then we obtain: 

\[
\sum_{t,s \in \F_{q}^{\times}}\sum_{(v,v^{\vee}) \in X(\F_q)} \psi\left[s^{-1}\left( \langle u, v^{\vee}\rangle  + \langle v, u^{\vee} \rangle\right)\right]\cdot\psi(t^{-1})
\cdot \psi\left[t^{-1}\left( \langle v, w^{\vee}\rangle  + \langle w, v^{\vee} \rangle\right)\right]\cdot\psi(s^{-1})
\]


\noindent which we may rewrite as:

\[
\sum_{t,s \in \F_{q}^{\times}}\sum_{(v,v^{\vee}) \in X(\F_q)} \psi\left[s^{-1}\left( \langle u, v^{\vee}\rangle  + \langle v, u^{\vee} \rangle + 1\right)\right]
\cdot \psi\left[t^{-1}\left( \langle v, w^{\vee}\rangle  + \langle w, v^{\vee} \rangle + 1\right)\right].
\]

Now we may sum with respect to $t$ and $s$ separately. We recall that $\sum_{t\in\F_q^{\times}}\psi(t^{-1}a)$ is either $-1$ if $a  \ne 0$ or $q-1$ if $a = 0$. So let us define a constructible function $c_{x, z}: X(\F_q) \to \C$ as follows:

\[
c_{x,z}(v,v^{\vee}) = \begin{cases}
    1 \text{ if }  \langle u, v^{\vee}\rangle  + \langle v, u^{\vee} \rangle \ne -1 \text{ and } \langle v, w^{\vee}\rangle  + \langle w, v^{\vee} \rangle  \ne -1\\
    -(q-1) \text{ if only one of }  \langle u, v^{\vee}\rangle  + \langle v, u^{\vee} \rangle \text{ or }  \langle v, w^{\vee}\rangle  + \langle w, v^{\vee} \rangle  \text{ equals }-1\\
    (q-1)^2 \text{ if }  \langle u, v^{\vee}\rangle  + \langle v, u^{\vee} \rangle = -1 \text{ and } \langle v, w^{\vee}\rangle  + \langle w, v^{\vee} \rangle  = -1.
\end{cases}
\]


\noindent Then we see that 

\[
\mathcal{F}^2(\delta_x)(z) = \sum_{y \in X(\F_q)} c_{x,z}(y).
\]


Geometrically we see that the two constraints $\langle u, v^{\vee}\rangle  + \langle v, u^{\vee} \rangle = -1$ and $\langle v, w^{\vee}\rangle  + \langle w, v^{\vee} \rangle  = -1$ correspond to affine hyperplanes in $V \times V^*$, while $X$ is a quadric hyperplane in $V \times V^*$. We are thus weighting the points of $X(\F_q)$ by $1$, $-(q-1)$, or $(q-1)^2$ depending on whether the point lies in zero, one, or two of these affine hyperplanes. 

Let us call the varieties -- affine hyperplanes -- cut out by $\langle u, v^{\vee}\rangle  + \langle v, u^{\vee} \rangle = -1$ and $\langle v, w^{\vee}\rangle  + \langle w, v^{\vee} \rangle  = -1$ by $\Pi_1$ and $\Pi_2$. Observe that replacing $x$ with $\lambda x$ results in $\langle \lambda u, v^{\vee}\rangle  + \langle v, \lambda u^{\vee} \rangle = -1$; or, equivalently $\langle  u, v^{\vee}\rangle  +  \langle v, u^{\vee} \rangle = -\lambda^{-1}$. This is a parallel translation of $\Pi_1$; let us call it $\Pi_1^{+\lambda^{-1}}$. The content of (\ref{KSum}) is that the number of $\F_q$-points in $\Pi_1(\F_q) \cap \Pi_2(\F_q) \cap X(\F_q)$ is equal to $\Pi_1^{+\lambda^{-1}}(\F_q)  \cap \Pi_2(\F_q) \cap X(\F_q)$ unless $\Pi_1 = \Pi_2$. (If $\Pi_1 = \Pi_2$, then $\Pi_1^{+\lambda^{-1}}$ is parallel, but not equal, to $\Pi_2$ -- hence the two have empty intersection.)

In any event, we observe that varieties $\Pi_1 \cap \Pi_2 \cap X$ are quadric hypersurfaces in affine space, so we may apply elementary means to count points. We find that:

\begin{equation}\label{casesfor}
\sum_{y \in X(\F_q)} c_{x,z}(y) = \begin{cases}
    q^{2d-1} + q^d - q^{d-1} \text{ if } x = z = 0\\
    q^{2d} - q^{2d-1} + q^d - q^{d-1} \text{ if } x = z, \text{ with neither } 0\\
    -q^{2d-1} + q^d - q^{d-1} \text{ if } x \ne z \text{ are proportional with neither } 0\\
    q^d -q^{d-1} \text{ if either one and only of } x, z = 0;\text{ or else  if }\\
    \,\,\,\,\,\,\, \,\,\,\,\,\,\,\,\,\,\,\,\,\,\,\,\,\, x \text{ and } z \text{ are not proportional and }  \langle u, w^{\vee} \rangle + \langle w, u^{\vee} \rangle = 0\\
    -q^{d-1} \text{ if } x \text{ and } z \text{ are not proportional and }  \langle u, w^{\vee} \rangle + \langle w, u^{\vee} \rangle \ne 0.
\end{cases}
\end{equation}


\noindent Observe the pleasant fact that, excepting the most singular stratum -- the origin $x = z= 0$ -- subtracting from a row the row below it yields a single power of $q$ (with sign). Moreover, these strata have been ordered so that each stratum in $X(\F_q)$ lies in the closure of the strata beneath it. For $\lambda \in \F_q^{\times}$, each stratum is invariant under the action of $(x, z) \mapsto (\lambda x, z)$, except for the $x=z$, $(x,z) \ne (0,0)$ stratum. If $x = z$, then $(\lambda x, z)$ (for $\lambda \ne 1$) lies in the stratum where $x \ne z$, neither are 0, and $x$ and $z$ are proportional. We conclude that

\begin{equation*}\label{FourDiracMon}
     \mathcal{F}^2(\delta_x - \delta_{\lambda x}) = q^{2d}(\delta_x - \delta_{\lambda x}),
\end{equation*}


\noindent as claimed.

\end{proof}

\noindent \textbf{Remark}. We see from (\ref{casesfor}) that the function space $\mathcal{S}$ is essentially the best possible space on which $q^{-d}\mathcal{F}$ is involutive. As mentioned previously, it appears to be analogous to functions coming from the category of ``regular monodromic sheaves" in \cite{BravermanPolishchuk1998}.  On the other hand, one can readily verify that the kernel used here is the finite field analogue of the kernel used in the Fourier transform for quadrics found in Gurevich-Kazhdan, Getz-Hsu-Leslie, and Getz \cite{GurKazh, GHL, GetzQuad}. Kobayashi-Mano \cite{KobMan}  use a kernel that is not analytic (given by the sum of a Bessel function and a finite sum of derivatives of Dirac masses at the origin). It is conceivable that one may remove the restriction on the function space if one alters the kernel along these lines; or, perhaps, if one includes certain ``boundary terms" as in \cite{GetzQuad}. Indeed, ibid. shows that the ``boundary terms" of the local field quadric Fourier transform are given by sums over smaller quadrics -- we have already seen the emergence of smaller quadrics in counting points on intersections of the quadric $X$ with pairs of affine hyperplanes.

\section{The Case of \texorpdfstring{$\SL_3$}{SL3}}

Up to isomorphism, we see that there are two kinds of nontrivial parabolic in $G$: the Borel (i.e., the 1+1+1-block parabolic), and the 2+1 block parabolic. The latter is already dealt with in Section \ref{Mirabolic}. 

\subsection{The Space \texorpdfstring{$\overline{\SL_3/U}$}{SL3/U}} We will ow describe the paraspherical space $\overline{G/U(B)}$, for $B$ the standard Borel. As with all Borels, this is the same as the BK-space $\overline{G/[B,B]}$. As we shall see, this case is the smallest-rank example for which $\overline{G/U}$ which is singular.

Observe that, like in the case of $\SL_2$, we have a map

\[
G/U(B) \,\,\,\longrightarrow \,\,\A^3
\]
\[
    \left(\begin{matrix}
        a & b & c\\
        d & e & f\\
        g & h & i
    \end{matrix}\right) \mapsto \left(\begin{matrix}
        a\\
        d\\
        g
    \end{matrix}\right).
\]


\noindent However, a simple dimension count shows that $G/U$ has dimension $8-3 = 5$, so this map cannot be an isomorphism onto its image. 

Observe that, were we to take the \textit{left} quotient by $U(B)$, we would have another map:

\begin{equation}\label{bottomrow}
    U(B)\backslash G \longrightarrow\A^3
\end{equation}
\[
\nonumber \left(\begin{matrix}
        a & b & c\\
        d & e & f\\
        g & h & i
    \end{matrix}\right) \mapsto \left(\begin{matrix}
        g & h & i
    \end{matrix}\right).
\]


\noindent Now we can define a map $G/U(B) \to \A^3$ given by inversion: $g U \mapsto Ug^{-1} \in U(B)\backslash G$. Thus we may postcompose inversion with the map (\ref{bottomrow}) to give us:

\begin{equation}\label{minormap}
    G/U(B) \longrightarrow \A^3
\end{equation}
\[
  \nonumber \left(\begin{matrix}
        a & b & c\\
        d & e & f\\
        g & h & i
    \end{matrix}\right) \mapsto \left(\begin{matrix}
        dh-ge & gb-ah & ae-db
    \end{matrix}\right).
\]


\noindent In other words, we have two invariants for $G/U(B)$: the first column of $g$ and the third row of $g^{-1}$ (otherwise known as the minors of the third column, or the cross-product of the first two columns; recall that $\det(g) =1$, which is why there are no denominators). Notice that the canonical pairing of these row-and-column vectors is always 0 (as the product would be the $(3,1)$ entry of $g^{-1}g = \Id$).

Thus, if we let $V = \A^3$ and $V^* = \A^3$, viewed as the dual vector space, we may define the map $\phi$:

\[
    G/U(B) \xrightarrow{\phi} V \times V^{*}
\]
\[    \left(\begin{matrix}
        a & b & c\\
        d & e & f\\
        g & h & i
    \end{matrix}\right) \mapsto\left( \left(\begin{matrix}
        a\\
        d\\
        g
    \end{matrix}\right), \left(\begin{matrix}
        dh-ge & gb-ah & ae-db
    \end{matrix}\right)\right),
\]


\noindent whose image is contained in the closed subset $\{(v, v^{\vee}) \in V \times V^* : \langle v, v^{\vee}\rangle = 0\}$, and consists of those pairs, subject to this constraint, such that neither $v^*$ nor $v$ is 0. Let

\begin{equation}\label{sl3affcl}
X := \{(v, v^{\vee}) \in V \times V^* : \langle v, v^{\vee}\rangle = 0\}
\end{equation}


\noindent and 

\[
X^* := \{ (v, v^{\vee}) \in X: v \ne 0 \text{ and } v^{\vee} \ne 0\}.
\]


\noindent Note that $X$ has dimension $6-1 = 5$, and is affine,  while $X - X^*$ is the union of two copies of $\A^3$, glued at the origin. Hence the complement of $X^*$ in $X$ is of codimension 2; thus, by Hartog's Lemma, any regular function on $X^*$ uniquely extends to $X$. Hence $X$ is the affine closure of $X^*$. Moreover, by a dimension count, we see that $\phi$ induces an isomorphism

\[
G/U(B) \xrightarrow{\sim} X^*
\]


\noindent so that 

\[
\overline{G/U(B)} \cong X.
\]


\noindent $X$ is a quadric hypersurface in $\A^6$, the affine cone over a quadric hypersurface in $\P^5$. It has an isolated, conical singularity at the origin -- since $\SL_2$ and the mirabolics of $\SL_3$ all have smooth paraspherical spaces, we see that this is the smallest example of a singular $\overline{G/U(P)}$. 

\subsection{The Slipper Pairing} Next, we consider the Slipper pairing for opposite Borels:

\[
\overline{G/U(B)}\times \overline{G/U(B^{\op})} \to \overline{T}.
\]


\noindent It is a well-known fact that the double-unipotent invariant functions of $\SL_n$ are generated by the determinants of the $i \times i$ square matrices located in the upper left corner of $\SL_n$, $i \le n-1$ \cite{GrossBook, NgoPCMI}.\footnote{For $i = n$, this is just the determinant of the whole matrix, which is 1 in the $\SL$ case. For $\GL_n$ this gives another double unipotent invariant. See Grosshans \cite{GrossBook}, Lemma 18.7, page 104.} In our case, the two invariants of 

\[
\left(\begin{matrix}
        a & b & c\\
        d & e & f\\
        g & h & i
    \end{matrix}\right)
\]


\noindent are $a$ and $ae - db$. These two functions are algebraically independent; thus the Wang monoid in this case is:

\[
\overline{T} \cong \A^2 \cong \Spec\, k[a, ae-bd].
\]


Now, we let $X^{\op} = \overline{G/U(B^{\op})} = \{(v, v^{\vee}) \in V \times V^*: \langle v, v^{\vee}\rangle = 0\}$, where 

\[
G/U(B^{\op}) \to X^{\op}
\]
\[
gU(B^{\op}) \mapsto (\text{third column of }g, \text{ first row of }g^{-1})
\]
\[
\left(\begin{matrix}
        a & b & c\\
        d & e & f\\
        g & h & i
    \end{matrix}\right)\left(\begin{matrix}
        1 & 0 & 0\\
        * & 1 & 0\\
        * & * & 1
    \end{matrix}\right) \mapsto \left( \left(\begin{matrix}
        c \\
        f \\
        i
    \end{matrix}\right), \left(\begin{matrix}
        ei-hf & hc-bi & bf-ec
    \end{matrix}\right)\right).
\]


\noindent The Slipper pairing is given by 

\[
\mathfrak{S}: X \times X^{\op} \to \A^2
\]
\[
\left((v, v^{\vee}),(w, w^{\vee})\right) \mapsto (\langle w, v^{\vee} \rangle, \langle v, w^{\vee} \rangle),
\]


\noindent since the upper right entry of $h^{-1}g$ is given by $(\text{first row of }h^{-1})\cdot(\text{first column of } g)$, while the determinant of the upper right $2\times 2$ is the same as the $(3,3)$-entry of $g^{-1}h$, which is $(\text{third row of }g^{-1})\cdot(\text{third column of } h)$. As we would expect, the pairing on the ``$G$-diagonal", in other words, $\mathfrak{S}(gU(B), gU(B^{\op}))$, is always $(1,1)$, the identity of the monoid $\A^2$.


\subsection{The Function \texorpdfstring{$J$}{J}} 
We now examine the function $J$ on the Wang monoid $\A^2$. To accomplish this we must analyze the adjoint representation 
of $M^{\vee} = T^{\vee}$ on $\mathfrak{u}^{\vee}$. We have 

\[X^*(T^{\vee}) = X_*(T) = \left\{a_1\lambda_1 + a_2\lambda_2 + a_3 \lambda_3 : a_1 + a_2 + a_3 = 0\right\}.
\]


\noindent Recall that $\SL_3^{\vee} = \PSL_3$. 
We see that the the adjoint action of $T^{\vee}$ in $\mathfrak{u}^{\vee}$ has weights $\lambda_1 - \lambda_2$, $\lambda_2 - \lambda_3$, and $\lambda_1 - \lambda_3$ 
(i.e., the positive coroots of $\SL_3$). Hence the Kloosterman sheaf (forgetting the Tate twist) on the torus $T$ of $\SL_3$ is given by the pull-push of the Artin-Schreier sheaf along the map:

\[
\begin{tikzcd}
             & \mathbb{G}_m^3 \arrow[ld, "\text{Tr}"'] \arrow[rd, "p_{\underline{\lambda}}"] &                                                    \\
\mathbb{A}^1 & {(t_1, t_2, t_3)} \arrow[ld, maps to] \arrow[rd, maps to]                     & T                                                  \\
t_1+t_2+t_3  &                                                                               & {\text{diag}(t_1t_3,\,t_2t_1^{-1},\,t_2^{-1}t_2^{-1})}
\end{tikzcd}
\]


\noindent Observe that the determinant of the $\text{diag}(t_1t_3,\,t_2t_1^{-1},\,t_2^{-1}t_2^{-1})$ is 1, so the image of $p_{\underline{\lambda}}$ does indeed lie in $T \subset \SL_3$.

Now, recall that our two double-unipotent invariant functions in $\overline{T}$ are the upper left corner, and the determinant of the upper left $2 \times 2$. Let us call these invariants $A$ and $B$, respectively. We find that $A(\text{diag}(t_1t_3,t_2t_1^{-1}t_2^{-1}t_2^{-1}))= t_1 t_3$ and $B(\text{diag}(t_1t_3,t_2t_1^{-1}t_2^{-1}t_2^{-1})) = t_2t_3$. Thus we find that $A + B = t_3(t_1+t_2)$, and that $t_1 + t_2 + t_3 = (A+B)t_3^{-1} + t_3$.

Therefore, as a function of $A$ and $B$ we may write 

\begin{equation}\label{DescFun}
p_{\underline{\lambda}!}\text{Tr}^*(\psi) \left(\text{diag}(a, b, c)\right) = \sum_{t_3\in \F_q^{\times}} \psi((A+B)t_3^{-1} + t_3) = \Kl(A+B),
\end{equation}


\noindent where $ac = A$, $bc =B$. We see that $\Kl(A+B)$ will be, up to a sign and power of $q$ discussed below, the function $J$ on the Wang monoid $\A^2$ (with coooridnates $A$ and $B$).

At the level of sheaves, say we let $\mathscr{K}_\psi$ denote denote the perverse sheaf obtained by the left-to-right pull-push of the Artin-Schreier sheaf $\mathcal{L}_\psi$ along the correspondence:

\begin{equation}\label{KlShDef}
\begin{tikzcd}
             & \mathbb{G}_m^2 \arrow[ld, "\sigma"'] \arrow[rd, "\varpi"] &              \\
\mathbb{G}_a &                                                        & \mathbb{G}_m,
\end{tikzcd}
\end{equation}


\noindent where the $\sigma$ is coordinate-wise addition and $\varpi$ is coordinate-wise multiplication and where we adopt the normalization $\mathscr{K}_\psi = \varpi_!\sigma^*(\mathcal{L}_\psi)[2]\left(1\right)$. We have that $\mathscr{K}_\psi$ is perverse and pure of weight 0 \cite{Del, KatzGauss}. 
(Notice that this is also the normalization given by \ref{GammaTorus}.) 
We consider its intermediate extension to $\A^1$, which we will (somewhat abusively) also call $\mathscr{K}_\psi$. We have that the corresponding function $\textrm{Fun}(\mathscr{K})[x] = q^{-1}\Kl(x)$. Then, letting $\Sigma: \A^2 \to \A^1$ be the sum of coordinates (which is smooth), $\Sigma^*(\mathscr{K}_\psi)[1]\left(\frac{1}{2}\right)$ is a perverse sheaf on $\A^2$, 
which furthermore agrees with $\Phi_{\underline{\lambda},T, \psi}$ when we pull back along $T \to \A^2$, $\text{diag}(a, b, c) \mapsto (A,B) = (a, ab)$. 

Thus applying the principle of perverse continuation \cite{NgoPCMI} we find that $\Sigma^*(\mathscr{K}_\psi)[1]\left(\frac{1}{2}\right)$ is our sheaf $\mathfrak{J}$ on the Wang Monoid $\A^2$. From this we deduce the formula for the function $J$:

\[
J(A,B) = -q^{-\frac{3}{2}} \Kl(A+B).
\]


\noindent Incorporating the Tate twist and dimension shift in (\ref{TransOp}), which (since $\dim \,\SL_3/U = 5$, $\dim \, T = 2$) introduce an additional factor of $-q^{-\frac{5-2}{2}} = -q^{\frac{3}{2}}$, we find that the Fourier transform is given by 

\begin{equation}\label{sl3opFour}
\mathcal{F}(f)(y) = q^{-3}\sum_{x \in X(\mathbb{F}_q)} f(x)K(x,y),
\end{equation}


\noindent where $K(x,y) = \Kl\left(\langle u, v^{\vee} \rangle + \langle v, u^{\vee}\rangle\right)$ (adopting as before, the notation $x = (u,u^{\vee})$ and $y = (v, v^{\vee})$).

We conclude that our Fourier transform for $\SL_3$ and opposite Borels is simply the Fourier transform (\ref{FourKl}) for a 3-dimensional space $V$. Thus the involutivity property for special functions on $\overline{\SL_3/U}$ follows from Theorem \ref{KlInv}.

\subsection{Normalized Intertwining Operators for \texorpdfstring{$\SL_3(\F_q)$}{SL3(Fq)}}\label{RFC} We will now discuss how our Fourier transform for opposite Borels in combination with the standard Braverman-Kazhdan Fourier transform for adjacent Borels give rise to normalized intertwining operators on $\SL_3(\F_q)$.

\subsubsection{The Geometry of Braverman-Kazhdan Pairings for \texorpdfstring{$\SL_3$}{SL3}} We reproduce, as a visual aid, Makisumi's image \cite{Mak} for the spherical apartment of $\SL_3$. Recall that each facet corresponds to a different standard parabolic; the corresponding parabolic has been superimposed on the image. The open 60-degree cones (the Weyl chambers) correspond to Borels, and Borels are adjacent if the corresponding Weyl chambers share a wall.

\begin{figure}[t]
\includegraphics[width=80mm]{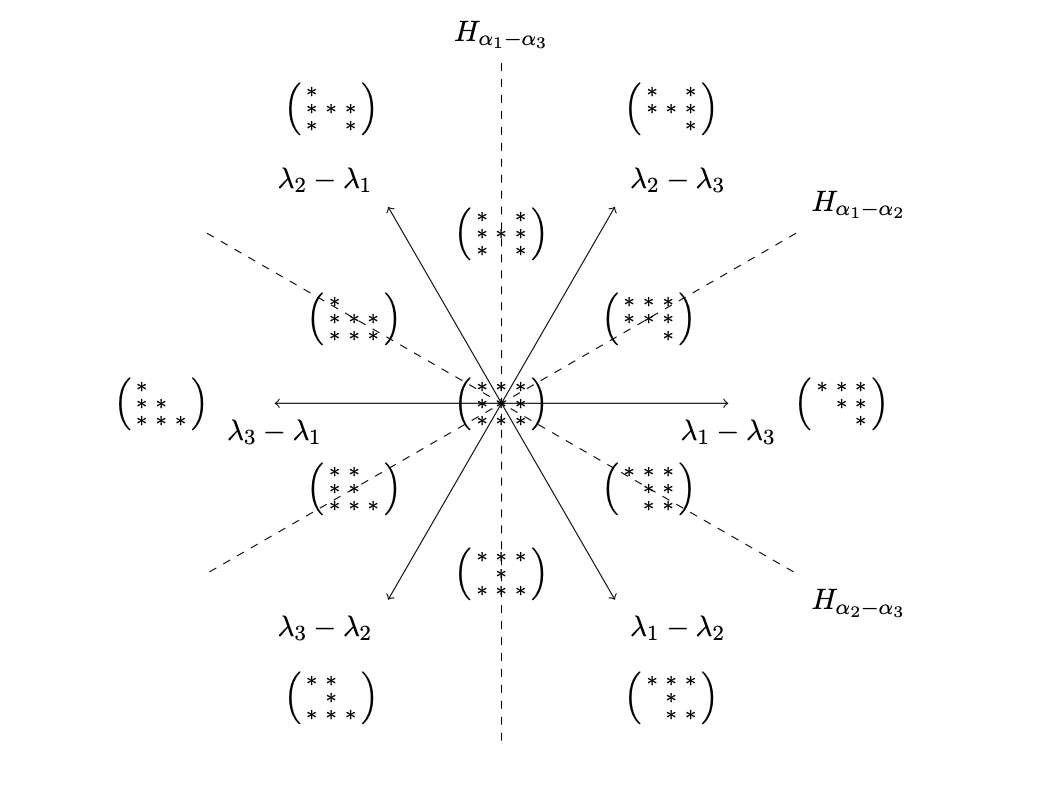}
\centering
\caption{The Spherical Appartment and Standard Parabolics for $\SL_3$  \cite{Mak}.}
\end{figure}


Now we will examine the Braverman-Kazhdan pairing for adjacent Borels in $\SL_3$. Let $B'$ be the Borel corresponding to the Weyl chamber containing $\lambda_2 - \lambda_3$ in the illsutration above. Let $X' = \overline{G/U(B')}$. We see that $X' = \{(v^*, v) \in V^* \times V : \langle v^*, v\rangle = 0\}$, with the map 

\[
G/U(B') \to X'
\]
\[
g U(B') \mapsto (\text{row 3 of } g^{-1},\text{ column 2 of }g)
\]
\[
   \left(\begin{matrix}
        a & b & c\\
        d & e & f\\
        g & h & i
    \end{matrix}\right)\left(\begin{matrix}
        1 & 0 & *\\
        * & 1 & *\\
        0 & 0 & 1
    \end{matrix}\right)  \mapsto\left(\left(\begin{matrix}
        dh-ge & gb-ah & ae-db
    \end{matrix}\right), \left(\begin{matrix}
        b\\
        e\\
        h
    \end{matrix}\right)\right).
\]


\noindent Now both $G/U(B)$ and $G/U(B')$ have natural maps to $G/[P,P]$, where $P$ is the 2+1 parabolic (in fact, $[P,P]$ is the subgroup of $G$ generated by $U(B)$ and $U(B')$). Thus there are maps between the corresponding affine closures; these are given by

\[
    X \to \A^3
\]
\[
    (v, v^{\vee}) \mapsto v^\vee
\]


\noindent and 

\[
    X' \to \A^3
\]
\[
    (v, v^{\vee}) \mapsto v^\vee.
\]


\noindent Hence we may consider the varieties

\begin{equation}\label{BarInc}
\overline{G/U(B)} \times_{\overline{G/[P,P]}} \overline{G/U(B')} = \{(v, v^{\vee}, w) \in V \times V^* \times V : \langle v, v^{\vee} \rangle = 0 \text{ and }\langle w, v^{\vee} \rangle =0\},
\end{equation}


\noindent and 

\begin{equation}\label{AffClTwid}
\widetilde{G/U(B)} \times_{G/[P,P]} \widetilde{G/U(B')} = \{(v, v^{\vee}, w) \in V \times V^* \times V : \langle v, v^{\vee} \rangle = 0, \,\langle w, v^{\vee} \rangle =0, \, v^\vee \ne 0\}.
\end{equation}


\noindent This latter space is a product of two rank-two vector bundles over $G/[P,P]$, and it possesses the Braverman-Kazhdan pairing

\begin{equation}\label{dualitysl3}
\langle-,-\rangle: \widetilde{G/U(B)} \times_{G/[P,P]} \widetilde{G/U(B')} \to \A^1.
\end{equation}


\noindent Explicitly, this is given as follows. Observe that $v \wedge w \in \bigwedge^2 (V) \cong V^*$. (We make this isomorphism canonical by insisting that $\bigwedge^3V$ has chosen basis $e_1 \wedge e_2 \wedge e_3$, giving us a specified isomorphism $\bigwedge^3V \cong \A^1$. Here the $e_i$ are the standard basis vectors of $V = \A^3$.) Since $v$ and $w$ both lie in $\ker(v^\vee)$, we see that $v \wedge w = \lambda v^\vee$. Then the map $(v,, v^{\vee}, w) \mapsto \lambda$ defines the map (\ref{dualitysl3}). When we identify all the vector spaces as $k^3$ with the standard dot product representing vector-covector evaluation, we may write the Braverman-Kazhdan pairing as:

\begin{equation}
(v, v^{\vee}, w) \mapsto \frac{v \times w}{v^\vee}
\end{equation}


\noindent where $\times$ means the classical 3-dimensional cross product. The ratio is meaningful because $v \times w$ and $v^\vee$ are, as we have seen, proportional.

Observe that the Braverman-Kazhdan Pairing (\ref{dualitysl3}) does \textit{not} algebraically extend to (\ref{BarInc}): the ratio $\lambda$ cannot be defined if $v^\vee =0$. This is why extending even the ``adjacent" Braverman-Kazhdan transformation to all functions on $\overline{G/U}(\F_q)$ is not so straightforward.

\subsubsection{Restricted Function Spaces and Composition of Transforms}\label{CompRes} We want to constrain our functions $f$ so that we may compose adjacent Braverman-Kazhdan transforms.

Firstly: we observe that the Braverman-Kazhdan transform corresponding the the flip across the $H_{\alpha_1 -\alpha_2}$ hyperplane in the spherical apartment, is defined on our model of $\overline{\SL_3/U}$ by:

\begin{equation}\label{FBKDef}
\mathcal{F}_{\text{BK}}(f) (v,v^{\vee}) = \sum_{u \in (v^{\vee})^{\perp}} f(u,v^{\vee})\psi\left(\frac{u \times v}{v^{\vee}}\right).
\end{equation}


\noindent We see that we must have 

\begin{equation}\label{v0}
f(v, 0) = 0
\end{equation}


\noindent if the transform is even to be well-defined. (Setting $f(v,0) = 0$ allows us to ignore the terms where we would have a division by 0 problem.) Similarly, if we want to be able to define the BK Fourier transform corresponding to the flip about $H_{\alpha_2 - \alpha_3}$, we will need to have

\begin{equation}\label{vee0}
    f(0, v^{\vee}) = 0
\end{equation}


\noindent for all $v^{\vee}$. Observe that these two constraints mean that $f$ is entirely supported on $G/U \subset \overline{G/U}$.

Next we demand that $\mathcal{F}_{\text{BK}} (f)$ vanish on the set $\{(v, v^{\vee}) : v = 0\}$; we see that:

\[
0 = \mathcal{F}_{\text{BK}}(f) (0,v^{\vee}) = \sum_{u \in (v^{\vee})^{\perp}} f(u,v^{\vee})\psi\left(\frac{u \times 0}{v^{\vee}}\right) =  \sum_{u \in (v^{\vee})^{\perp}} f(u,v^{\vee}).
\]


\noindent Hence 

\begin{equation}\label{const1}
    \sum_{u \in (v^{\vee})^{\perp}} f(u,v^{\vee}) = 0
\end{equation}


\noindent for all $v^{\vee}$. Likewise, considering the $H_{\alpha_2-\alpha_3}$ flip, we see that we want 

\begin{equation}\label{const2}
    \sum_{v^{\vee} \in u^{\perp}} f(u,v^{\vee}) = 0
\end{equation}


\noindent for all $u$.

We define our restricted function space $\mathcal{S}'$ as follows:

\begin{align}\label{S'Ddef}
    \mathcal{S}'(X(\F_q), \C) = \Biggl\{f : \forall& \, v, v^{\vee} \text{ with }\langle v, v^{\vee}\rangle  = 0,
      \sum_{\lambda \in \F_q^{\times}} f(\lambda v, v^{\vee}) = 0, \\
     \nonumber&\sum_{\lambda \in \F_q^{\times}} f(v, \lambda v^{\vee}) =0, \text{ and }
     \ \sum_{\lambda \in \F_q^{\times}} f(\lambda v, \lambda v^{\vee})=0 \Biggr\}.
\end{align}


\noindent Functions which are regular monodromic in the sense of \cite{BravermanPolishchuk1998} satisfy these $\G_m$-averaging properties. Observe that this automatically implies (\ref{v0}) and (\ref{vee0}). Moreover, we see that these constraints also imply (\ref{const1}) and (\ref{const2}). In fact, the constraints imposed by $\mathcal{S}'$ are somewhat stronger than those given by (\ref{v0}), (\ref{vee0}),  (\ref{const1}) and (\ref{const2}). It is possible we may be over-constraining $\mathcal{S}'$; however, we will eventually use the full force of the constraints on $\mathcal{S}'$ in comparing compositions of BK transforms to the Kloosterman transform. Finally, observe that $\mathcal{F}_{\text{BK}}$ does indeed preserve the space $\mathcal{S}'$. For this all these reasons, we have no obstruction to composing BK transforms.

We now wish to compose several of these BK transforms and compare it with our Kloosterman transform for opposite Borels. To simplify notation, we will let $c_i$ represent the $i$th column vector of a matrix $g$ in $\SL_3$, and let $\overline{r}_i$ represent the $i$th row vector of $g^{-1}$. We will compose the three flips along the upper 180-degree arc in the spherical apartment illustrated above, going from the upper-triangular Borel to the lower-triangular. We first tabulate the invariants for each $G/U_w$ we will encounter en route:


\begin{center}
\begin{tabular}{||c c c c||} 
 \hline
$\left.\SL_3 \middle/\left(\begin{matrix}
     1 & * & *\\
     0 & 1 & *\\
     0 & 0 & 1
 \end{matrix}\right),\right.$ &$\left.\SL_3 \middle/\left(\begin{matrix}
     1 & 0 & *\\
     * & 1 & *\\
     0 & 0 & 1
 \end{matrix}\right),\right.$ & $\left.\SL_3 \middle/\left(\begin{matrix}
     1 & 0 & 0\\
     * & 1 & *\\
     * & 0 & 1
 \end{matrix}\right),\right.$ & $\left.\SL_3 \middle/\left(\begin{matrix}
     1 & 0 & 0\\
     * & 1 & 0\\
     * & * & 1
 \end{matrix}\right)\right.$ \\ 
 \hline\hline
 $c_1$ & $c_2$ & $c_2$ & $c_3$ \\ 
 \hline
 $\overline{r_3}$ &  $\overline{r_3}$ &  $\overline{r_1}$ &  $\overline{r_1}$ \\
 \hline
\end{tabular}
\end{center}


Therefore, we may write the composite transform as:

\begin{equation}\label{TotTrans}
\mathcal{F}_{\text{comp}}(f)(c_3, \overline{r_1}) =
\sum_{\overset{c_1}{c_1\cdot \overline{r_3}=0}}
\sum_{\overset{\overline{r_3}}{c_2\cdot \overline{r_3}=0}}
\sum_{\overset{c_2}{c_2\cdot \overline{r_1}=0}}
f(c_1, \overline{r_3})
\psi\left(\frac{c_1 \times c_2}{\overline{r_3}}\right)
\psi\left(\frac{\overline{r_3} \times \overline{r_1}}{c_2}\right)
\psi\left(\frac{c_2 \times c_3}{\overline{r_1}}\right),
\end{equation}


\noindent where the vectors over which we are summing are presumed to be nonzero. We want to compare this with the Kloosterman transformation:

\begin{equation}\label{KlooTrans}
\mathcal{F}(f)(c_3, \overline{r_1}) : = \sum_{(c_1, \overline{r_3}) \in X(\F_q)} f(c_1, \overline{r_3}) K[(c_1, \overline{r_3}), (c_3, \overline{r_1})]
\end{equation}


\noindent where $K[(c_1, \overline{r_3}), (c_3, \overline{r_1})]:= \Kl(c_1 \cdot \overline{r_1} + c_3 \cdot \overline{r_3})$.


\begin{prop}\label{SumCompat}
    For all $f \in \mathcal{S}'$, we have

    \[
    \mathcal{F}_{\text{comp}}(f) = \mathcal{F}(f).
    \]
\end{prop}


\begin{proof}
If we fix $c_1, \overline{r_3}$ (and of course $c_3$ and $\overline{r_1}$ are fixed to begin with!), we may find the total coefficient of $f(c_1, \overline{r_3})$ in the sum (\ref{TotTrans}):

\[
\sum_{\overset{c_2}{c_2 \in \overline{r_3}^{\perp}\cap\overline{r_1}^{\perp}}}
\psi\left(\frac{c_1 \times c_2}{\overline{r_3}}\right)
\psi\left(\frac{\overline{r_3} \times \overline{r_1}}{c_2}\right)
\psi\left(\frac{c_2 \times c_3}{\overline{r_1}}\right).
\]


\noindent If $\overline{r_3}$ and $\overline{r_1}$ are not proportional, then $c_2$ lives in the 1-dimensional subspace $\overline{r_1}^{\perp} \cap \overline{r_3}^{\perp}$: in fact, it is precisely the subspace generated by $\overline{r_3} \times \overline{r_1}$. Let $c_2 = \lambda \cdot \overline{r_3} \times \overline{r_1}$. We now reveal the centerpiece of our argument -- Lagrange's formula for the classical cross-product:

\begin{equation}
    \textbf{a} \times (\textbf{b} \times \textbf{c}) = (\textbf{a} \cdot \textbf{c})\textbf{b} - (\textbf{a}\cdot\textbf{b})\textbf{c}.
\end{equation} 


\noindent This shows that $c_1 \times c_2 = \lambda (c_1 \cdot \overline{r_1})$, since we assume that $c_1 \cdot \overline{r_3} = 0$. Likewise, $c_2 \times c_3 = \lambda(c_3 \cdot \overline{r_3})$. We manifestly have $\frac{\overline{r_3} \times \overline{r_1}}{c_2} = \lambda^{-1}$, so we find that the coefficient of $f(c_1, \overline{r_3})$ is:

\[
\sum_{\lambda \in \F_q^{\times}} \psi( [c_1 \cdot \overline{r_1} + c_3 \cdot \overline{r_3}]\lambda + \lambda^{-1} ) = \Kl(c_1 \cdot \overline{r_1} + c_3 \cdot \overline{r_3}).
\]


\noindent This is precisely what we want: ``generically", the transform $\mathcal{F}_{\text{comp}}$ resulting from the composite of the three BK transforms gives the same kernel as our Kloosterman Fourier transform. More precisely, if $\overline{r_3}$ and $\overline{r_1}$ are independent, then the total coefficient of $f(c_1, \overline{r_3})$ in the expansion of $\mathcal{F}_{\text{comp}}(f)$ in (\ref{TotTrans}) is exactly $K((c_1, \overline{r_3}), (c_3, \overline{r_1})$, the coefficient in its expansion in (\ref{KlooTrans}). 

We have shown that, after amalgamating all the summands with a factor of $f(c_1, \overline{r_3})$ in (\ref{TotTrans}), we will obtain a sum that ``generically" matches (\ref{KlooTrans}) term-wise. We must now examine the remaining ``boundary" terms. We will show that the two ``boundary" sums agree (though only ``in totality" -- not necessarily term-wise). More precisely, we will show that the sum of all the terms in (\ref{TotTrans}) and (\ref{KlooTrans}) for which $\overline{r_3}$ and $\overline{r_1}$ are dependent is, in both cases, 0. This is where we make critical use of the assumption that $f$ belongs to $\mathcal{S}'$.

Let us consider the remaining ``boundary terms" of the sum (\ref{TotTrans}). In this case, $\overline{r_1}$ and $\overline{r_3}$ are proportional, and $c_2$ lives in the 2-dimensional space $\overline{r_1}^{\perp} = \overline{r_3}^{\perp}$. Let $\overline{r_3} = \mu \overline{r_1}$. We see that the sum of such terms is:

\[
\sum_{\overset{c_1}{c_1\cdot \overline{r_1}=0}}
\sum_{\mu \in \F_q^{\times}}
\sum_{\overset{c_2}{c_2\cdot \overline{r_1}=0}}
f(c_1, \mu\overline{r_1})
\psi\left(\frac{c_1 \times c_2}{\mu\overline{ r_1}}\right)
\psi\left(\frac{\mu\overline{r_1} \times \overline{r_1}}{c_2}\right)
\psi\left(\frac{c_2 \times c_3}{\overline{r_1}}\right)
\]


\noindent The second $\psi$ factor disappears because $\overline{r_1} \times \overline{r_1} \ = 0$. And now we may apply the substitution $c_1 \to \mu c_1$, which gives us

\[
\sum_{\overset{c_1}{c_1\cdot \overline{r_1}=0}}
\sum_{\mu \in \F_q^{\times}}
\sum_{\overset{c_2}{c_2\cdot \overline{r_1}=0}}
f(\mu c_1, \mu\overline{r_1})
\psi\left(\frac{c_1 \times c_2}{\overline{ r_1}}\right)
\psi\left(\frac{c_2 \times c_3}{\overline{r_1}}\right)
\]


\noindent Since the last two factors are independent of $\mu$, we find the whole expression is 0 because $f \in \mathcal{S}'$. Hence for $f \in \mathcal{S}'$, the total sum of those terms in (\ref{TotTrans}) for which $\overline{r_3}$ is proportional to $\overline{r_1}$ is 0.

Now we examine the total sum of those terms in (\ref{KlooTrans}) for which $\overline{r_3}$ is proportional to $\overline{r_1}$. This is:

\begin{align*}
    &\sum_{c_1 \in \overline{r_1}^{\perp}}\sum_{\overline{r_3}= \mu \overline{r_1}} f(c_1, \overline{r_3})  K[(c_1, \overline{r_3}), (c_3, \overline{r_1})]\\
    &= \sum_{c_1 \in \overline{r_1}^{\perp}}\sum_{\lambda, \mu \in \F_q^{\times}} 
    f(c_1, \mu \overline{r_1})  
    \psi[\lambda( c_1 \cdot \overline{r_1} + \mu c_3 \cdot \overline{r_1}) + \lambda^{-1}].
\end{align*}


\noindent But $c_3\cdot \overline{r_1} = 0$ by assumption, whence the above becomes:

\[
\sum_{c_1 \in \overline{r_1}^{\perp}}\sum_{\mu \in \F_q^{\times}} 
    f(c_1, \mu \overline{r_1})  
    \cdot \Kl(c_1 \cdot \overline{r_1}) = 0,
\]


\noindent since $f \in \mathcal{S}'$.

Thus the total sum of those terms in (\ref{KlooTrans}) for which $\overline{r_3}$ is proportional to $\overline{r_1}$ is 0, and so agrees with the total sum of such terms in (\ref{TotTrans}). Since we have already matched up the other terms in the two summations, the proposition follows.

\end{proof}

Putting Proposition \ref{SumCompat} together with the involutivities of $\mathcal{F}_{BK}$ and our Kloosterman $\mathcal{F}$, we may now state the following theorem:

\begin{thm}
    Let $G = \SL_3$, $T$ the standard torus. Let $\mathcal{B}_T$ denote the set of (six) Borels of $G$ cotaining $T$. For each $B \in \mathcal{B}_T$, we note that that $G/U(B)$ is a $G \times T$-variety. 
    
    For each Borel $B \in \mathcal{B}_T$, let $(\Phi^{\vee})^+_B$ denote the collection of positive coroots for $B$. For each coroot $\alpha^{\vee}$, let $T_{\alpha^\vee}$ denote the corresponding 1-dimensional subtorus. Let us define the space
    
    \[
    \mathcal{S}'(\overline{SL_3/U(B)}(\F_q), \C)
    \]

    
    \noindent as the collection of $\C$-valued functions on $\overline{G/U(B)}(\F_q)$ whose average under the right action of $T_{\alpha^{\vee}}$  is 0 for each $\alpha^{\vee} \in(\Phi^{\vee})^+_B$. 
    
    i) The transforms $q^{-1}\mathcal{F}_{BK_{B',B}}: \mathcal{S}'(\overline{SL_3/U(B)}(\F_q) \to \mathcal{S}'(\overline{SL_3/U(B')}(\F_q)$  between adjacent Borels $B$ and $B'$ induced by fiberwise Fourier transforms (and explicitly defined in \ref{FBKDef}) generate a family of normalized intertwining operators between any pair of $B$, $B' \in \mathcal{B}_T$.
    
    ii) Moreover, the intertwiner for opposite Borels agrees with $q^{-3}\mathcal{F}$, where $\mathcal{F}$ is defined as in (\ref{FourSlip}). 
    \end{thm}

    Viewing everything as functions on the set $X(\F_q)$, for the variety

    \[
    X = \{(v, v^{\vee}) \in V \times V^* : \langle v, v^{\vee}\rangle = 0\},
    \]
    

    \noindent where $V = \A^3$, we may rewrite Theorem 7.2 as a single $S_3$-action: 
    
    \begin{thm} There exists a $G$-equivariant action of $S_3 = W_T(\SL_3)$ on the space $\mathcal{S}'(X(\F_q), \C)$ as defined in (\ref{S'Ddef}); the fnite-field analogue of the the Gelfand Graev action. This action is $G$-equivariant, and $T$-equivariant under the right $T$-action twisted by $W$. It is generated by the two transformations:

    \begin{equation}
    \mathcal{F}_{\alpha_1-\alpha_2}(f)(v,v^{\vee}) = q^{-1}\sum_{u  \in (v^{\vee})^{\perp}} f(u,v^{\vee})\psi\left(\frac{u \times v}{v^{\vee}}\right)
    \end{equation}


     \noindent and 

      \begin{equation}
    \mathcal{F}_{\alpha_2-\alpha_3}(f)(v,v^{\vee}) = q^{-1}\sum_{u^{\vee}  \in v^{\perp}} f(v,u^{\vee})\psi\left(\frac{u^{\vee} \times v^{\vee}}{v}\right)
    \end{equation}


    \noindent while the longest element of $W$ acts via the quadratic Kloosterman Fourier transform:

    \begin{equation}\label{sl3op}
\mathcal{F}_{w_0}(f)(v,v^{\vee}) = q^{-3}\sum_{(u,u^{\vee}) \in X(\mathbb{F}_q)} f(u, u^{\vee})K[(u,u^{\vee}),(v, v^{\vee})]
\end{equation}


\noindent where $K(x,y) = \Kl\left(\langle u, v^{\vee} \rangle + \langle v, u^{\vee}\rangle\right)$.

\end{thm}

\section{The case of \texorpdfstring{$\Sp_4$}{SP4}} 

In this section, we let $G = \Sp_4$. 

\subsection{Conventions} We will adopt the convention in Makisumi:

\[
\Sp_4 = \left\{ g \in \GL_4 \,: \,g^{\text{T}} \left(\begin{matrix}
     &  &  & 1\\
     &  & -1 & \\
     & 1 &  & \\
    -1 &  &  & 
\end{matrix}\right) g = \left(\begin{matrix}
     &  &  & 1\\
     &  & -1 & \\
     & 1 &  & \\
    -1 & &  & 
\end{matrix}\right) \right\}.
\]


\noindent This slightly unusual convention is superior to the usual $\left(\begin{matrix}
    0 & \Id_2\\
    -\Id_2 &  0
\end{matrix}\right)$ insofar as it permits us to let the standard Borel be given by upper triangular matrices. The maximal torus is given by matrices of the form:

\[
\left(\begin{matrix}
    t_1  & & & \\
    & t_2 & & \\
    & & t_2^{-1} & \\
    & & & t_1^{-1}
\end{matrix}\right)
\]


\noindent The projections $\alpha_1$ and $\alpha_2 : T \to \mathbb{G}_m$ give us a system of fundamental weights for $\Sp_4$; the set of roots is:

\[
\Phi = \{\pm 2\alpha_1, \, \pm 2\alpha_2, \, \pm\alpha_1 \pm \alpha_2\}
\]


\noindent while the set of coroots is given by:

\[
\Phi^{\vee} = \{\pm \lambda_1,\, \pm \lambda_2, \, \pm \lambda_1 \pm \lambda_2\}.
\]


\noindent Further details may be found in Makisumi. 

We will, throughout this section, let $\omega$ denote the symplectic form preserved by the action of $\Sp_4$:

\[
\omega\left(\left(\begin{matrix}
    a\\
    b\\
    c\\
    d
\end{matrix}\right), \left(\begin{matrix}
    a'\\
    b'\\
    c'\\
    d'
\end{matrix}\right)\right) = ad' - bc'+ cb' - da'
\]

\subsection{The Spherical Apartment}

\begin{figure}
    \includegraphics[width=80mm]{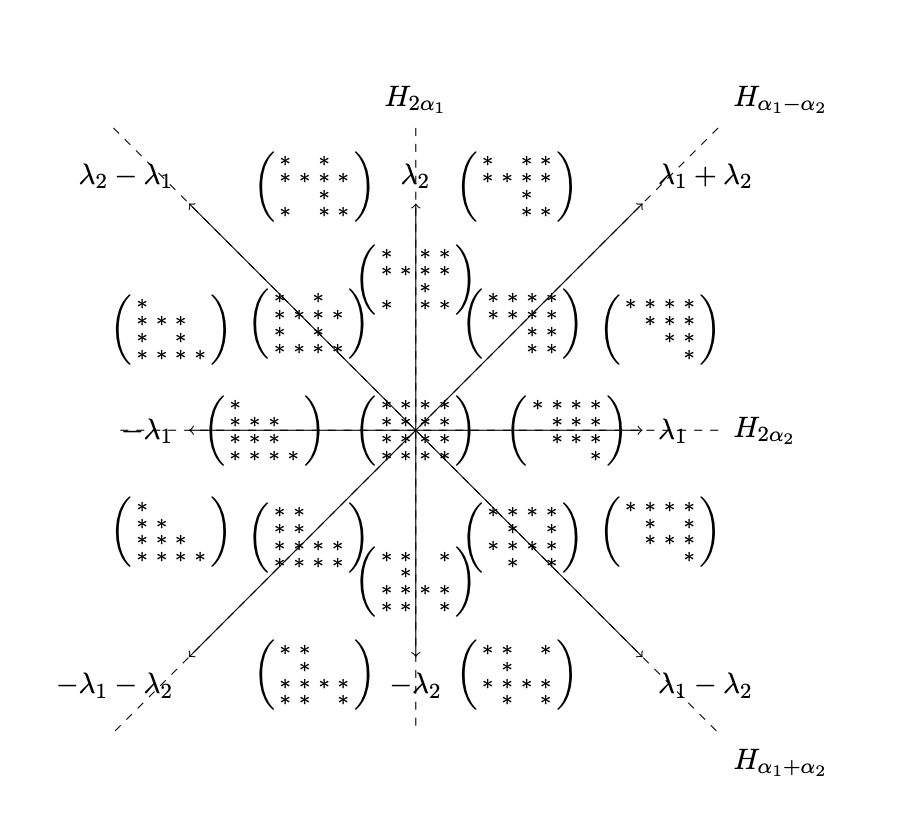}
    \centering
    \caption{The Spherical Apartment and Coroots for $\Sp_4$ \cite{Mak}.}
\end{figure}


\noindent We see that there are, up to isomorphism, 3 kinds of nontrivial parabolic subgroups:

\noindent 1) the Borel: $\left(\begin{matrix}
    * & * & * & *\\
    0 & * & * & *\\
    0 & 0 & * & *\\
    0 & 0 & 0 & *
\end{matrix}\right),$ 2) the Klingen parabolic: $\left(\begin{matrix}
    * & * & * & *\\
    0 & * & * & *\\
    0 & * & * & *\\
    0 & 0 & 0 & *
\end{matrix}\right)$ and  3) the Siegel parabolic:$\left(\begin{matrix}
    * & * & * & *\\
    * & * & * & *\\
    0 & 0 & * & *\\
    0 & 0 & * & *
\end{matrix}\right).$


\noindent We will center our attention on the Siegel Parabolic, which we shall call $P$. then $
U(P) = \left\{ \left(\begin{matrix}
    1 & 0 & * & *\\
    0 & 1 & * & *\\
    0 & 0 & 1 & 0\\
    0 & 0 & 0 & 1
\end{matrix}\right)\right\}$; note that the upper-right $2 \times 2$ matrix is actually three-dimensional, and $U(P)$ is isomorphic to the Abelian group $\A^3$: $\left(\begin{matrix}
     * & *\\
    * & *\\
\end{matrix}\right) = \left(\begin{matrix}
     a & b\\
     c & -a\\
\end{matrix}\right)$.

 The Levi $M$, which we shall consider both as a subgroup of $\Sp_4$ and as a quotient of $P$, is isomorphic to $\GL_2$, and is embedded in $\Sp_4$ via: $\left\{ \left(\begin{array}{@{}c|c@{}}
  m & 0 \\
\hline
0 & m^{-1}
\end{array}\right)\right\}.$


\subsection{Geometry of the Affine Closure} Now we shall describe the paraspherical space $\overline{G/U(P)}$. We observe there is a map:

\[
\Sp_4/U(P) \to \Mat_{4,2}
\]

\[
\left(
  \begin{array}{cccc}
    \vertbar & \vertbar &    \vertbar    & \vertbar \\
    v_{1}    & v_{2}    &  v_3 & v_{4}    \\
    \vertbar & \vertbar &   \vertbar     & \vertbar 
  \end{array}
\right)
\left(
\begin{matrix}
  1 & 0 & * & *\\
  0 & 1 & * & *\\
  0 & 0 & 1 & 0\\
  0 & 0 & 0 & 1
  \end{matrix}
 \right) \mapsto \left( \begin{array}{cc}
    \vertbar & \vertbar     \\
    v_{1}    & v_{2}         \\
    \vertbar & \vertbar 
  \end{array}\right).
\]

\noindent In the above, the $v_i$ denote column vectors. The image of this map is the set of $(v_1, v_2)$ such that $\omega(v_1, v_2) = 0$ (since they come from the first two columns of an element of $\Sp_4$), and such that $v_1$ and $v_2$ are linearly independent. The closure of this image in $\Mat_{4,2}$ drops the linear independence constraint; we let

\[
X := \{(v_1, v_2) \in \Mat_{4,2} \,: \, \omega(v_1, v_2)= 0\}
\]


\noindent denote this closure. Note that $X$ has dimension 7. Since $G/U(P)$ also has dimension $7 =10-3$, we can verify that $G/U(P) \to X$ is an isomorphism onto its image. Moreover, the complement of the this image (i.e., the complement of the rank-2 locus in $X \subset \Mat_{4,2}$) has dimension 5; codimension 2 in $X$. Thus, Hartog's Lemma implies that $X$ is isomorphic to the affine closure of $G/U(P)$. 

Observe that $X$ is a quadric hypersurface in $\A^8$; it is the affine cone over a quadric hypersurface in $\P^7$. Hence it has an isolated conical singularity at the origin.

\subsection{The Slipper Pairing} We will now consider the Slipper pairing. As we did for $G/U(P)$, we provide a model for $G/U(P^{\op})$ via:

\[
\Sp_4/U(P^{\op}) \to \Mat_{4,2}
\]

\[
\left(
  \begin{array}{cccc}
    \vertbar & \vertbar &    \vertbar    & \vertbar \\
    v_{1}    & v_{2}    &  v_3 & v_{4}    \\
    \vertbar & \vertbar &   \vertbar     & \vertbar 
  \end{array}
\right)
\left(
\begin{matrix}
  1 & 0 & 0 & 0\\
  0 & 1 & 0 & 0\\
  * & * & 1 & 0\\
  * & * & 0 & 1
  \end{matrix}
 \right) \mapsto \left( \begin{array}{cc}
    \vertbar & \vertbar     \\
    v_{3}    & v_{4}         \\
    \vertbar & \vertbar 
  \end{array}\right).
\]


\noindent Now, observe that $U(P^{\op}) \times U(P)$-invariants of $\Sp_4$ are given by:

\[
 \left(\begin{matrix}
    1 & 0 & 0 & 0\\
    0 & 1 & 0 & 0\\
    * & * & 1 & 0\\
    * & * & 0 & 1
\end{matrix}\right) 
\left(\begin{matrix}
    a & b & c & d\\
    e & f & g & h\\
    i & j & k & l\\
    m & n & o & p
\end{matrix}\right) \left(\begin{matrix}
    1 & 0 & * & *\\
    0 & 1 & * & *\\
    0 & 0 & 1 & 0\\
    0 & 0 & 0 & 1
\end{matrix}\right) = \left(\begin{matrix}
    a & b & * & *\\
    e & f & * & *\\
    * & * & * & *\\
    * & * & * & *
\end{matrix}\right).
\]


\noindent Thus the Wang monoid $\overline{M}$ of the Siegel parabolic $P$ is given by $\Mat_2 \cong \Spec 
\, k[a, b, e, f]$.

Also, observe that in $\Sp_4$:

\[
\left(\begin{matrix}
    * & * & c & d\\
    * & * & g & h\\
    * & * & k & l\\
    * & * & o & p
\end{matrix}\right)^{-1} = \left(\begin{matrix}
    -p & l & -h & d\\
    o & -k & g & -c\\
    * & * & * & *\\
    * & * & * & *
\end{matrix}\right).
\]


\noindent Thus the Slipper pairing is given by:

\[
\mathfrak{S}: G/U(P) \times G/U(P^{\op}) \to \overline{M}
\]
\[
(gU(P), hU(P^{\op})) \mapsto U^{\op}h^{-1}g U
\]
\begin{align*}
\left((v_1, v_2),(w_1, w_2)\right) \mapsto &\left(\begin{matrix}
    -\omega(v_1, w_2) & -\omega(v_2, w_2)\\
    \omega(v_1, w_1) & \omega(v_2, w_1) 
\end{matrix}\right)\\
 =&\left(\begin{matrix}
    0 & -1\\
    1 & 0
\end{matrix}\right)\left(\begin{matrix}
    \omega(v_1, w_1) & \omega(v_2, w_1)\\
    \omega(v_1, w_2) & \omega(v_2, w_2) 
\end{matrix}\right).
\end{align*}


\subsection{The Function \texorpdfstring{$J$}{J}}\label{Sp4Func} Next we compute the function $J$. The first thing we must do is diagonalize the adjoint representation of $M^{\vee}$ on $\mathfrak{u}^{\vee}$. We see that $G^{\vee} \cong \SO_5$, $M^{\vee} \cong \GL_2$, and the adjoint representation of $M^{\vee}$ on $\mathfrak{u}^{\vee}$ is $\Std \oplus \det$. The corresponding cocharacters are thus given by $\lambda_1$, $\lambda_2$ and $\lambda_1 + \lambda_2$. We see that the function associated to the corresponding hypergeometric sheaf on $T$ (neglecting the power of $q$ coming from the Tate twist) is:

\begin{align*}
\Psi\left(\begin{matrix}
    t_1 & 0\\
    0 & t_2
\end{matrix}\right) &= \sum_{\substack{s_1d = t_1\\s_2 d = t_2}}\psi(s_1 +s_2 + d)\\
&= \sum_{d}\psi((t_1 +t_2)/d + d)\\
&= \Kl(\tr (t))
\end{align*}


\noindent where $\Kl(a) = \sum_{d \in \F_q^{\times}}\psi(a/d + d)$ is the usual Kloosterman sum.

Applying a similar argument to the case of $\SL_3$ (see (\ref{KlShDef})), we find that if $\Tr: \Mat_{2,2} \to \A^1$ is the trace map then the perverse sheaf $\Tr^*(\mathscr{K}_\psi)[3]\left(\frac{3}{2}\right)$ which agrees with the $\gamma$-sheaf on the regular 
semisimple locus of $\GL_2 \subseteq \Mat_2$. 
Applying the principle of perverse continuation \cite{NgoPCMI}, 
we conclude that $\Tr^*(\mathscr{K}_\psi)[3]\left(\frac{3}{2}\right)$ is our sheaf $\mathfrak{J}_{\text{Mat}_2}$ 
on the Wang Monoid. Our function on $M$ is $J(m) = -q^{-\frac{5}{2}}\Kl(\tr(m))$. 

The Tate twist in the transform (\ref{TransOp}) introduces an additional factor of $-q^{-\frac{7-4}{2}} = -q^{3/2}$, since $\dim \, \Sp_4/U = 7$ and $\dim \, M = 4$. This gives a total overall factor of $q^{-4}$. We find that the function on $\overline{\Sp_4/U}\times\overline{\Sp_4/U^{\op}}$ corresponding to our Fourier kernel $\mathfrak{S}^*(\mathfrak{J})\left(\frac{n}{2}\right)[-m]$ is given by

\[
\left((v_1, v_2),(w_1, w_2)\right) \mapsto q^{-4} \Kl(\omega(v_1, w_2) + \omega(w_1, v_2)).
\]


\noindent Thus our Fourier transforms between opposite Siegel Parabolics in $\Sp_4$ reduces to the quadratic Kloosterman Fourier transform from Section 6. Involutivity (for ``special" functions) then follows from Theorem \ref{KlInv}.

\section*{Data availability}

Data sharing is not applicable to this article, as no datasets were generated or analyzed during the current study.

\section*{Competing interests}

The author has no competing interests to declare that are relevant to the content of this article.

\bibliography{refs}{}
\bibliographystyle{alpha}

\end{document}